\documentclass[10pt]{amsart}

\usepackage{amssymb, amsmath,amsthm}
\usepackage{mathrsfs}
\usepackage{comment}
\usepackage{mathtools}
\usepackage{color}
\usepackage{caption}
\usepackage{subcaption}
\usepackage{graphics}
\usepackage{graphicx}
\usepackage{enumitem}

\mathtoolsset{showonlyrefs}

\newtheorem{thm}{Theorem}[section]

\newtheorem{prop}[thm]{Proposition}
\newtheorem{cor}[thm]{Corollary}
\newtheorem{lem}[thm]{Lemma}
\theoremstyle{definition}
\newtheorem{rem}[thm]{Remark}
\newtheorem{def1}[thm]{Definition}

\newcommand{\ra}{\rightarrow}

\newcommand{\mc}{\mathcal}

\newcommand{\mb}{\mathbb}
\newcommand{\sg}{\sigma}

\newcommand{\vp}{\varphi}
\newcommand{\R}{\mb{R}}
\newcommand{\C}{\mb{C}}
\newcommand{\N}{\mb{N}}
\newcommand{\Z}{\mb{Z}}
\newcommand{\E}{\mb{E}}
\newcommand{\F}{\mc{F}}
\newcommand{\G}{\mc{G}}
\newcommand{\U}{\mc{U}}
\renewcommand{\ss}{\substack}

\newcommand{\e}{\varepsilon}

\newcommand{\mbf}{\boldsymbol}
\newcommand{\asum}{\sideset{}{^{\ast}}\sum}
\renewcommand{\bar}{\overline}
\renewcommand{\bmod}{\text{ mod }}

\newcommand{\bfree}{\mathbf{1}_{B\text{-free}}}
\newcommand{\Nbfree}{N_{B\text{-free}}}
\newcommand{\meanb}{\mathcal{M}_B}
\newcommand{\sgb}{\langle B\rangle}
\newcommand{\sgbsf}{[B]}
\newcommand{\mobb}{\mu_B}
\newcommand{\lcm}{\mathrm{lcm}}

\newcommand{\intervalind}{\mbf{1}_{(0,1]}}

\title{Squarefrees are Gaussian in short intervals}
\author{Ofir Gorodetsky, Alexander P. Mangerel, Brad Rodgers}

\newcommand{\Addresses}{{
  \footnotesize
  \bigskip
  \footnotesize

	\textsc{Mathematical Institute, University of Oxford, Oxford, OX2 6GG, UK}\par\nopagebreak
	\textit{E-mail address:} \texttt{ofir.goro@gmail.com}

	\medskip

	\textsc{Department of Mathematical Sciences, Durham University, Durham, DH1 3LE, UK}\par\nopagebreak
	\textit{E-mail address:} \texttt{smangerel@gmail.com}

	\medskip

	\textsc{Department of Mathematics and Statistics, Queen's University, Kingston, Ontario, K7L 3N6, Canada}\par\nopagebreak
	\textit{E-mail address:} \texttt{brad.rodgers@queensu.ca}

}}
	
\date{}

\numberwithin{equation}{section}

\begin{document}

\begin{abstract}
We show that counts of squarefree integers up to $X$ in short intervals of size $H$ tend to a Gaussian distribution as long as $H\rightarrow\infty$ and $H = X^{o(1)}$. This answers a question posed by R.R. Hall in 1989. More generally we prove a variant of Donsker's theorem, showing that these counts scale to a fractional Brownian motion with Hurst parameter $1/4$. In fact we are able to prove these results hold in general for collections of $B$-free integers as long as the sieving set $B$ satisfies a very mild regularity property, for Hurst parameter varying with the set $B$.
\end{abstract}

\maketitle

\section{Introduction}

\subsection{Statistics of counts of squarefrees}

Let $S \subset \N$ be the set of squarefree natural numbers (that is natural numbers without a repeated prime factor; by convention we include $1 \in S$). We write \[N_S(x) = |\{n \leq x:\; n \in S\}|\]
for the number of squarefrees no more than $x$. It is well known that $N_S(x) \sim x/\zeta(2)$, thus the squarefrees have asymptotic density $1/\zeta(2) = 6/\pi^2$. Our purpose in this note is to investigate their distribution at a finer scale. In particular we will investigate the distribution of squarefrees in a random interval $(n,n+H]$, where $n$ is an integer chosen uniformly at random from $1$ to $X$, with $X\rightarrow\infty$.

If $H$ is fixed and does not grow with $X$, at most $O(1)$ squarefrees can lie in such an interval. Their distribution as $X\rightarrow\infty$ is slightly complicated but completely understood; it may be described by Hardy--Littlewood type correlations which can be derived from elementary sieve theory (see \cite{Mir1, Mir2}). Or, more abstractly, the distribution of squarefrees in an interval of size $O(1)$ can be described by a non-weakly mixing stationary ergodic process (see \cite{Cellarosi2013, sarnak2011three}).

For $H$ tending to infinity with $X$ matters become at once simpler and more difficult; simpler because some of the irregularities in the distribution just described are smoothed out at this scale, but more difficult in that natural conjectures become more difficult to prove.

Let \[N_S(n,H) = N_S(n+H)-N_S(n)\]
be the count of squarefrees in the interval $(n,n+H]$. R.R. Hall \cite{Hall1, Hall2} was the first to investigate the distribution of this count when $H$ grows with $X$. In \cite[Corollary~1]{Hall1}, Hall proved that the variance of the number of squarefrees is of order $\sqrt{H}$ if $H$ is not too large with respect to $X$. More exactly, as $X\rightarrow\infty$ we have
\begin{equation}
\label{eq:variance_asymptotic}
\frac{1}{X} \sum_{n\le X}  \left( N_S(n,H) - \frac{H}{\zeta(2)} \right)^2 \sim A\sqrt{H}, \qquad A = \frac{\zeta\left(\frac{3}{2}\right)}{\pi} \prod_{p \text{ prime}} \left( 1-\frac{3}{p^2} + \frac{2}{p^3}\right)
\end{equation}
as long as $H \to \infty$ with $H \le X^{2/9-\varepsilon}$. 

Keating and Rudnick \cite{Keating2016} studied this problem in a function field setting, connecting it with Random Matrix Theory, and suggested based on this that \eqref{eq:variance_asymptotic} will hold for $H \le X^{1-\varepsilon}$. The best known result is \cite{GMRR}, where it is shown that \eqref{eq:variance_asymptotic} holds for $H \le X^{6/11-\varepsilon}$ unconditionally and $H \le X^{2/3-\varepsilon}$ on the Lindel\"of Hypothesis. In fact in \cite{GMRR} it is shown that even an upper bound of order $\sqrt{H}$ for $H \le X^{1-\varepsilon}$ for all $\varepsilon > 0$ would already imply the Riemann Hypothesis.

Because $N_S(n,H)$ is on average of order $H$, one might naively have expected the variance to also be of order $H$. That the variance is of order $\sqrt{H}$ speaks to the fact that the squarefrees are a rather rigid sequence. This can be discerned even visually in comparison, for instance, to the primes (Figure ~\ref{fig:primes_squarefrees}) and we will return to give a more exact description of it in Section \ref{subsec:fBm}.

In \cite{Hall2} Hall\footnote{Note that our definition of $M_k$ differs slightly from \cite{Hall2}. Hall does not normalize by the factor $1/X$.} studied higher moments of counts of squarefrees in short intervals
\begin{equation}
M_{k}(X,H)=\frac{1}{X} \sum_{n\le X} \left(N_S(n,H) - \frac{H}{\zeta(2)} \right)^k,
\end{equation}
where $k$ is a positive integer, proving the upper bound $\lim_{X\rightarrow\infty} M_k(X,H) \ll_k H^{(k-1)/2}.$ Various authors have asked whether this can be refined, with the most recent result, 
\[ M_k(X,H) \ll_\varepsilon H^{k/4+\varepsilon},\]
for any $\varepsilon > 0$ as long as $H \le X^{4/(9k) - \varepsilon}$, being due to Nunes \cite{Nunes1}. For $H$ in the range considered, this is an optimal upper bound up to the factor of $H^\varepsilon$. In \cite{Avdeeva2016} extensive numerical evidence is presented that suggests that these moments are in fact Gaussian.

Our first main result confirms this conjecture.
\begin{thm} \label{thm:GaussianMoments}
For $1 \leq H \le X$, as $X\rightarrow\infty$,
\[M_k(X,H)  = \mu_k (A^{\frac{1}{2}} H^{\frac{1}{4}})^k + O_k(H^{\frac{k}{4}-\frac{c}{k}} + H^k X^{-\frac{1}{3} + o(1)}),\]
for every positive integer $k$, where $\mu_k = 1\cdot 3 \cdots (k-1)$ if $k$ is even and $\mu_k = 0$ if $k$ is odd. Here $c > 0$ is an explicit constant.
\end{thm}

Thus if $H\rightarrow\infty$ and $H \le X^{4/(9k)-\varepsilon}$ for some $\varepsilon > 0$, we have \[M_k(X,H) = (\mu_k + o(1)) (A^{\frac{1}{2}} H^{\frac{1}{4}})^k\text{ as }X\rightarrow\infty.\] Note that the main term is the $k$-th moment of a centered Gaussian random variable with variance $A\sqrt{H}$.

If $H = X^{o(1)}$, then for any fixed $k$ we have that $H$ satisfies $H \leq X^{4/(9k)-\varepsilon}$ for some $\varepsilon > 0$ for sufficiently large $X$. Hence by the moment method (see \cite[Section 30]{Billingsley1995}), we obtain the following result.
\begin{thm}\label{thm:sqDistGauss}
Let $H = H(X)$ satisfy \[H\ra \infty, \text{  yet  }\,\frac{\log H}{\log X} \ra 0\text{ as }X \ra \infty.\]
Then, for any $z \in \mb{R}$,
\[\lim_{X \ra \infty} \frac{1}{X}\left|\left\{n \leq X : \frac{N_S(n,H) - \frac{H}{\zeta(2)}}{A^{\frac{1}{2}}H^{\frac{1}{4}}} \geq  z\right\}\right| = \frac{1}{\sqrt{2\pi}} \int_z^{\infty} e^{-\frac{t^2}{2}} dt.\]
\end{thm}

That is, the centered, normalized counts tend in distribution to a Gaussian random variable.

Gaussian limit theorems are known for the sums over short intervals of several important arithmetic functions (for instance divisor functions $d_k$ (see \cite{Lester2016}) and the sums-of-squares representation function $r$ (see \cite{Hughes2004})), and Gaussian limit theorems for counts of primes in short intervals are known under the assumption of strong versions of the Hardy--Littlewood conjectures \cite{MS}, but Theorem \ref{thm:sqDistGauss} seems to be the first instance of an \emph{unconditional} proof of Gaussian behavior for short interval counts of a non-trivial, natural number theoretic sequence.

\begin{figure}
\subcaptionbox{A short interval near $100000$ containing $125$ primes.}
{\scalebox{.37}{\includegraphics{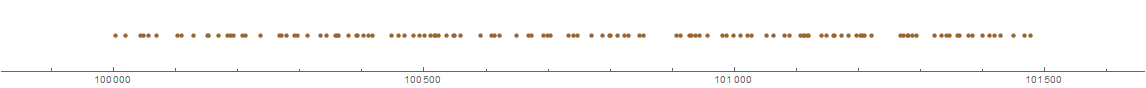}}}

\subcaptionbox{A short interval near $100000$ containing $125$ squarefrees.}
{\scalebox{.37}{\includegraphics{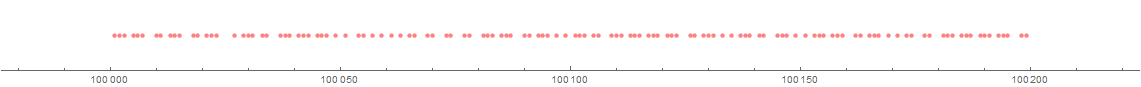}}}

\caption{A comparison between a short interval containing 125 primes with a short interval containing 125 squarefrees. The relative paucity of gaps and clusters of squarefrees is indicative of the rigidity of their distribution.} 
\label{fig:primes_squarefrees}
\end{figure}

Hall in \cite{Hall2} asked also about the order of magnitude of the absolute moments
\[ M_\lambda^+(X;H) = \frac{1}{X} \sum_{n \leq X} \left| N_S(n,H) - \frac{H}{\zeta(2)}\right|^\lambda,\]
and as a standard corollary of Theorems \ref{thm:GaussianMoments} and \ref{thm:sqDistGauss} we obtain an asymptotic formula for these.

\begin{cor}\label{cor:absMoments}
For fixed $\lambda > 0$, let $H = H(X)$ satisfy \[H \ra \infty,\text{ yet }\, \,\frac{\log H}{\log X} \ra 0\text{ as }X \ra \infty.\]
Then
\[ M_\lambda^+(X,H) \sim \frac{2^{\frac{\lambda}{2}}}{\sqrt{\pi}} \Gamma\left(\frac{\lambda+1}{2}\right) (A \sqrt{H})^{\frac{\lambda}{2}}.\]
\end{cor}

\begin{proof}
We give a quick derivation in the language of probability. By \cite[Theorem 25.12]{Billingsley1995} and the subsequent corollary, if $Y_j$ is a sequence of random variables tending in distribution to a random variable $Y$ and $\sup_j \E\, |Y_j|^{1+\varepsilon} < +\infty$ for some $\varepsilon > 0$, then $\E\,Y_j \ra \E\, Y$. Having chosen the function $H = H(X)$, for each $X$ let $n \in [1,X]$ be chosen randomly and uniformly and define the random variables $\Delta_X = |(N_S(n,H) - H/\zeta(2))/(A^{1/2}H^{1/4})|^\lambda$. Then as $X\rightarrow\infty$, we have that $\Delta_X$ tends in distribution to $|G|^\lambda$ for $G$ a standard normal random variable, by Theorem \ref{thm:sqDistGauss}. Moreover, Theorem \ref{thm:GaussianMoments} implies for any even integer $2\ell > \lambda$ that $\limsup_{X \ra \infty} \E |\Delta_X|^{2\ell} < +\infty$. Thus the result follows by computing $\E\, |G|^\lambda$ via calculus.
\end{proof}

\begin{rem}\label{rem:AP}
For a given result relating to the behavior of an arithmetic function in short intervals, it is natural to consider the analogous problem in a short arithmetic progression. For example, in analogy to the quantity $N_S(n,H)$ for $n \in [1,x]$ and $H = H(x)$ slowly growing, one might consider the quantity 
\[ N_S(x;q,a) := |\{n \leq x : n \in S, n \equiv a \bmod{q}\}|,\]
where $1 \leq q \leq x$ is chosen so large that $x/q$ (which corresponds in this context to $H$) is only slowly growing, and $a \bmod{q}$ is a specified residue class. When $a$ is coprime to $q$ the expected size of $N_S(x;q,a)$ is \[\frac{1}{\zeta(2)}\frac{x}{q} \prod_{p|q}\left(1-\frac{1}{p^2}\right),\] and one can define the analogous moments
\[\tilde{M}_k(x;q) := \frac{1}{\phi(q)} \asum_{a \bmod{q}} \left(N_S(x;q,a) - \frac{x}{q\zeta(2)} \prod_{p|q}\left(1-\frac{1}{p^2}\right)\right)^k \text{ for $k \geq 2$}.\]
As noted by Nunes \cite[Sections 1.2, 3.2]{Nunes1}, one may reduce the estimation of $\tilde{M}_k(x;q)$ to a quantity that is very similar to what is obtained in the course of estimating $M_k(x,H)$, making the analysis of the problem for arithmetic progressions nearly identical to that of short intervals. As such, one could very similarly obtain an arithmetic progression analogue of the Gaussian limit theorem Theorem \ref{thm:sqDistGauss}, if desired. However, unlike the short interval problem the problem in progressions does not seem to admit a nice interpretation in the language of fractional Brownian motion (see Section \ref{subsec:fBm}). We have therefore chosen to focus on the short interval problem in this paper in order to avoid making this paper even longer.
\end{rem}

\subsection{B-frees}

It is natural to write our proofs in the more general setting of $B$-free numbers. We recall their definition shortly, but first we fix some notation.

For a sequence $J$ of natural numbers we will write $\mbf{1}_J$ for the indicator function of $J$, and
\[N_J(x) = \sum_{n\leq x} \mbf{1}_J(n)\]
for the count of elements of $J$ no more than $x$.

\begin{def1}\label{def:index}
We say that a sequence $J$ is of \emph{index} $\alpha \in [0,1]$ if $N_J(x) = x^{\alpha+o(1)}$.
\end{def1}

\begin{def1}\label{def:varying}
A measurable function $L$ defined, finite, positive, and measurable on $[K,\infty)$ for some $K \geq 0$ is said to be \emph{slowly varying} if for all $a > 0$,
\[ \lim_{x\ra \infty} \frac{L(ax)}{L(x)} = 1.\]

A sequence $J \subset \N$ is said to be \emph{regularly varying} if $N_J(x) \sim x^\alpha L(x)$ for some $\alpha \in [0,1]$ and some slowly varying function $L$.
\end{def1}

For instance the function $L(x) = (\log x)^j$ is slowly varying, for any $j \in \R$. For any slowly varying function $L$ it is necessary that $L(x) = x^{o(1)}$, but this condition is not sufficient. Clearly in the definition above $\alpha$ will be the index of the regularly varying sequence $J$. Further information on regularly varying sequences can be found in \cite[Chapter 4.1]{PolyaSzego1}

Fix a non-empty subset $B \subset \N_{>1}$ of pairwise coprime integers with $\sum_{b \in B} 1/b < \infty$; we call such a set a \emph{sieving set}. We say that a positive integer is \emph{$B$-free} if it is indivisible by every element of $B$. 
For instance if $B=\{p^2:\; p \text{ prime}\}$, $B$-frees are nothing but squarefrees. Another studied example is $B=\{p^m: \; p \text{ prime}\}$ for some $m \ge 3$, for which $B$-frees are the $m$-th-power free numbers, i.e., integers indivisible by an $m$-th power of a prime. The notion of $B$-frees was introduced by Erd\H{o}s \cite{Erdos1966}, who was motivated by Roth's work \cite{Roth1951} on gaps between squarefrees. (See also \cite{Brudern2009,Jancevskis2009} for the closely related notion of \emph{convergent sieves}. We also mention that \cite{Jancevskis2009} upper bounded the quantity $\tilde{M}_2(x;q)$ mentioned in Remark~\ref{rem:AP}.)

We write $\bfree \colon \N \to \C$ for the indicator function of $B$-free integers, and $\Nbfree(x)$ for the number of $B$-free integers $n\leq x$. For the sets $B$ we are considering here, it is known (see e.g., \cite[Theorem 4.1]{ElAbdalaoui2015}) that
\begin{equation}\label{eq:meanvalue}
\Nbfree(x) \sim \meanb x, \quad \text{for}\quad \meanb = \prod_{b \in B} \left(1-\frac{1}{b}\right) \in (0,1),
\end{equation}
so that $B$-frees have asymptotic density $\meanb$.

We write $\sgb$ for the multiplicative semigroup generated by $B$, that is the set of positive integers that can be written as a product of (possibly repeated) elements of $B$. By convention $1 \in \sgb$, as $1$ arises from the empty product. (For instance if $B=\{p^2:\; p \text{ prime}\}$, then $\sgb = \{n^2:\; n \in \N\}$.)

We introduce an arithmetic function $\mobb \colon \N \to \C$, analogous to the M\"{o}bius function $\mu$, defined by
\[ \mobb(n) = \begin{cases} 0 & \text{if }n \notin \sgb \text{ or if there exists }b \in B \text{ with } b^2 \mid n,\\ (-1)^k & \text{if }n = b_1 b_2 \cdots b_k, \, b_1<b_2<\ldots<b_k, \, b_j \in B.\end{cases}\]
Observe that $\mobb$ and $\bfree$ are multiplicative and relate via the multiplicative convolution \begin{equation}\label{eq:dirichconv}
\bfree = \mobb \ast \mathbf{1}.
\end{equation}
We denote by $\sgbsf \subset \sgb$ the subset of $\sgb$ of elements $n$ satisfying $\mobb^2(n)=1$, i.e., those $n \in \sgb$ such that no $b \in B$ divides $n$ twice. (For instance, if $B = \{p^2:\; p\; \textrm{prime}\}$, then we have $\sgbsf = \{n^2:\; n \; \textrm{squarefree}\}$.) We will often use without mention the fact that both $\sgbsf$ and $\sgb$ are closed under $\gcd$ and $\lcm$, or equivalently, they are sublattices of the positive integers with respect to these two operations.

\subsubsection{Variance and Moments}

Let $\Nbfree(n,H) = \Nbfree(n+H)-\Nbfree(n)$ be the count of $B$-frees in an interval $(n,n+H]$. We consider the moments
\[ M_{k,B}(X,H)=\frac{1}{X} \sum_{n\le X} \left(\Nbfree(n,H) - \meanb H \right)^k.\]

\begin{prop}\label{prop:limiting_moments}
If the sequence $\sgb$ has index $\alpha \in (0,1)$, then for each fixed positive integer $k$,
\[ C_{k,B}(H) := \lim_{X\ra \infty} M_{k,B}(X,H), \]
exists and moreover for any $\e > 0$ we have
\begin{align}
M_{k,B}(X,H) = C_{k,B}(H) + O_k(H^k X^{\tfrac{\alpha-1}{\alpha+1} + \e} + 1).
\end{align}
\end{prop}

We will describe an explicit formula for $C_{k,B}(H)$ later.

When $k = 2$ we have the following general result for the variance, building on ideas of Hausman and Shapiro \cite{Hausman} and Montgomery and Vaughan \cite{MV},

\begin{prop}\label{prop:index_variance}
If the sequence $\sgb$ is of index $\alpha \in (0,1)$, then
$C_{2,B}(H) = H^{\alpha+o(1)}$.
\end{prop}

If $\sgb$ is in addition a regularly varying sequence then we prove an asymptotic formula for $C_{2,B}(H)$.

\begin{prop}\label{prop:regularlyvarying_variance}
If the sequence $\sgb$ is regularly varying with index $\alpha \in (0,1)$, then
\[ C_{2,B}(H) \sim A_\alpha N_{\sgb}(H),\]
where
\begin{equation}\label{eq:eulerprod_var}
A_\alpha := \zeta(2-\alpha) \gamma(\alpha) \prod_{b\in B} \left(1 - \frac{2}{b} + \frac{2}{b^{1+\alpha}} - \frac{1}{b^{2\alpha}}\right) \quad \text{with} \quad \gamma(\alpha): = \frac{(2\pi)^\alpha}{\pi^2} \cos\left(\frac{\pi \alpha}{2}\right)\Gamma(1-\alpha).
\end{equation}
\end{prop}

This generalizes a result of Avdeeva \cite{Avdeeva2015} which requires more robust assumptions about $\sgb$. It also gives a new proof for \eqref{eq:variance_asymptotic} that does not use contour integration and is essentially elementary.

In fact, we do not need an asymptotic formula for the variance to prove that the moments are Gaussian.

\begin{thm}\label{thm:main b}
If $\sgb$ has index $\alpha \in (0,1)$, then
\begin{equation}\label{eq:gen b res}
C_{k,B}(H) = \mu_k C_{2,B}(H)^{\frac{k}{2}} + O_k(H^{\frac{k\alpha}{2}-\frac{c}{k}})
\end{equation}
for every positive integer $k$. Here $c$ is an absolute constant depending only on $\alpha$.
\end{thm}

It is evident that Proposition \ref{prop:index_variance} and Theorem \ref{thm:main b} recover the moment estimate Theorem \ref{thm:GaussianMoments} for squarefrees. Moreover, for the same reasons as given for the central limit theorem there, we have the following theorem.

\begin{thm}\label{thm:bdistGauss}
Let $H = H(X)$ satisfy 
\[H \ra \infty,\text{ yet }\, \,\frac{\log H}{\log X} \ra 0\text{ as }X\ra \infty.\] If the sequence $\sgb$ has index $\alpha \in (0,1)$, then for any $z \in \R$,
\[ \lim_{X \ra \infty} \frac{1}{X}\left|\left\{n \leq X : \frac{\Nbfree(n,H) - \meanb H}{\sqrt{C_{2,B}(H)}} \geq  z\right\}\right| = \frac{1}{\sqrt{2\pi}} \int_z^{\infty} e^{-\frac{t^2}{2}} dt.\]
\end{thm}

\begin{rem}\label{rem:gaussian_limited_H}
Combining Theorem \ref{thm:main b} and Proposition \ref{prop:limiting_moments}, we see that for each $k$, the moments $M_{k,B}(X,H)$ will be asymptotically Gaussian as long as $H, X\ra \infty$ with \[H \le X^{\frac{c_{\alpha}}{k}-\varepsilon},\text{ where  }c_{\alpha} = \frac{1-\alpha}{(1+\alpha)(1-\frac{\alpha}{2})}.\]
\end{rem}
In Section \ref{subsec:longGaps} we give a further application of Theorem \ref{thm:main b} to estimates for the frequency of long gaps between consecutive $B$-free numbers, improving results of Plaksin \cite{Plaksin1990} and Matom\"{a}ki \cite{Matomaki2012}.

\subsection{Fractional Brownian motion}\label{subsec:fBm}

We have mentioned that the squarefrees and more generally $B$-frees in a random interval $(n,n+H]$ with $n\leq X$ are governed by a stationary ergodic process if $H$ remains fixed. The number of $B$-frees in such an interval is $\meanb H$ on average. This process has measure-theoretic entropy which becomes smaller the larger $H$ is chosen to be, and thus does not seem very `random'. This may be compared to primes in short intervals $(n, n + \lfloor \lambda \log X\rfloor ]$, which contain on average $\lambda$ primes. In such intervals primes are conjectured to be distributed as a Poisson point process, and thus appear very `random.'

Nonetheless a glance at Figure \ref{fig:primes_squarefrees} comparing squarefrees to primes -- along with consideration of the central limit theorems we have just discussed -- reveals that at a scale of $H\rightarrow\infty$, $B$-frees still retain some degree of randomness. It turns out that there is a natural framework to describe the `random' behavior of $B$-frees at this scale (analogous to the Poisson process for primes above) and this is fractional Brownian motion. 

We give here a short introduction to fractional Brownian motion, as we believe this perspective sheds substantial light on the distribution of $B$-frees; however, the remainder of the paper is arranged so that a reader only interested in the central limit theorems of the previous sections can avoid this material.

\begin{def1}\label{def:fBm}
A random process $\{Z(t) : t \in [0,1]\}$ is said to be a fractional Brownian motion with Hurst parameter $\gamma \in (0,1)$ if $Z$ is a continuous-time Gaussian process which satisfies $Z(0) = 0$ and also satisfies $\E\, Z(t) = 0$ for all $t \in [0,1]$ and has covariance function
\begin{equation}\label{eq:covar_fBm}
\E\, Z(t) Z(s) = \frac{1}{2}\left(t^{2\gamma} + s^{2\gamma} - |t-s|^{2\gamma}\right),
\end{equation}
for all $t, s \in [0,1].$
\end{def1}

Using $Z(0) = 0$, it is easy to see the covariance condition \eqref{eq:covar_fBm} is equivalent to
\begin{equation}\label{eq:covar_fBm_alt}
\E\, |Z(t)-Z(s)|^2 = |t-s|^{2\gamma}.
\end{equation}
For a proof that such a stochastic process exists and is uniquely defined by this definition see e.g. \cite{Nourdin}.

Classical Brownian motion is a fractional Brownian motion with a Hurst parameter of $\gamma = 1/2$. If $\gamma > 1/2$, increments of the process are positively correlated, with a rise likely to be followed by another rise, while if $\gamma < 1/2$, increments of the process are negatively correlated.

Donsker's theorem is a classical result in probability theory showing that a random walk with independent increments scales to Brownian motion (see \cite[Section 8]{Billingsley99}). We prove an analogue of Donsker's theorem for counts of $B$-frees using the following set up. We select a random starting point $n \le X$ at uniform and define the random variables $\xi_1, \xi_2, \ldots$ in terms of $n$ by
\[ \xi_k = \bfree(n+k) - \meanb = \begin{cases} 1 - \meanb & \text{if $n+k$ is $B$-free,} \\
-\meanb & \text{otherwise}. \end{cases}\]
Set
\begin{equation}\label{eq: partial sum interpol}
Q(\tau) = \sum_{k=1}^{\lfloor \tau \rfloor} \xi_k + \{\tau\} \xi_{\lfloor \tau \rfloor + 1},
\end{equation}
where $\{\tau\}$ denotes the fractional part.

For integer $\tau$ this is a random walk which increases on $B$-frees and decreases otherwise, and for non-integer $\tau$, the function $Q(\tau)$ linearly interpolates between values; thus $Q(\tau)$ is continuous.

\begin{thm}\label{thm:fBm_main}
Let $H = H(X)$ satisfy \[H\ra \infty, \text{ yet }\,\,\frac{\log H}{\log X} \ra 0\text{ as }X \ra \infty,\] and choose a random integer $n \in [1,X]$ at uniform. Suppose that $\sgb$ is a regularly varying sequence of index $\alpha \in (0,1)$ and define the function
\begin{equation}\label{eq:W_def}
W_X(t) = \frac{1}{\sqrt{A_\alpha N_{\sgb}(H)}}Q(t\cdot H),
\end{equation}
where $A_\alpha$ is defined by \eqref{eq:eulerprod_var}. Then, as a random element of $C[0,1]$, the function $W_X$ converges in distribution to a fractional Brownian motion with Hurst parameter $\alpha/2$ as $X\rightarrow\infty$.
\end{thm}

\begin{figure}[ht]
    \centering
	\begin{subfigure}{0.45\textwidth}
        \includegraphics[width=\textwidth]{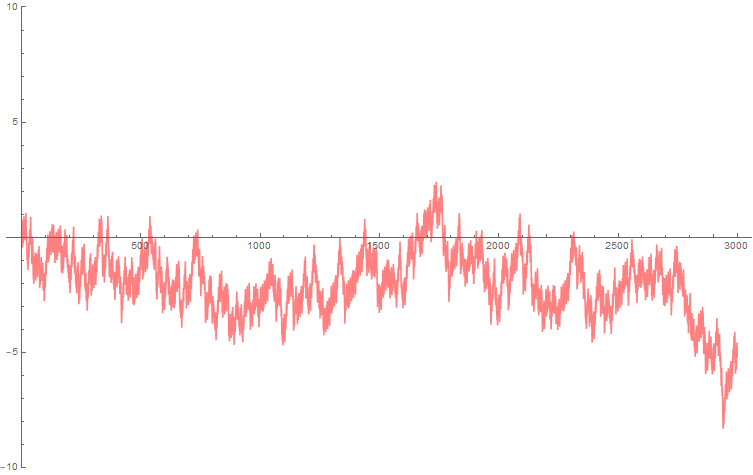}
        \caption{A random walk on squarefrees}
    \end{subfigure}
	\qquad
    \begin{subfigure}{0.45\textwidth}
        \includegraphics[width=\textwidth]{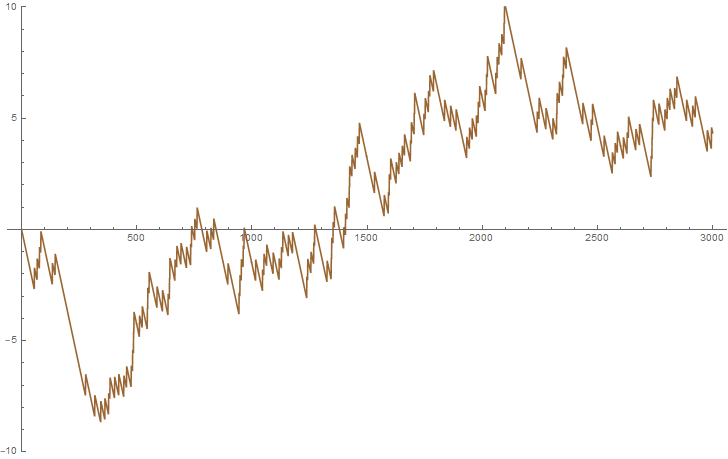}
        \caption{A random walk on primes}
    \end{subfigure}
    
	\caption{A graph of $Q(\tau)$ and an analogue for the primes in the short interval $(n,n+H]$ with $n = 875624586$ and $H = 3000$. (a) illustrates a walk which increments by $1 - 6/\pi^2$ if $u$ is squarefree, and $-6/\pi^2$ otherwise. (b) illustrates a walk which increments by $1 - 1/\log u$ if $u$ is prime, and $-1/\log u$ if $u$ is composite.}
\end{figure}

\begin{figure}[ht]
    \centering
	\begin{subfigure}{0.45\textwidth}
        \includegraphics[width=\textwidth]{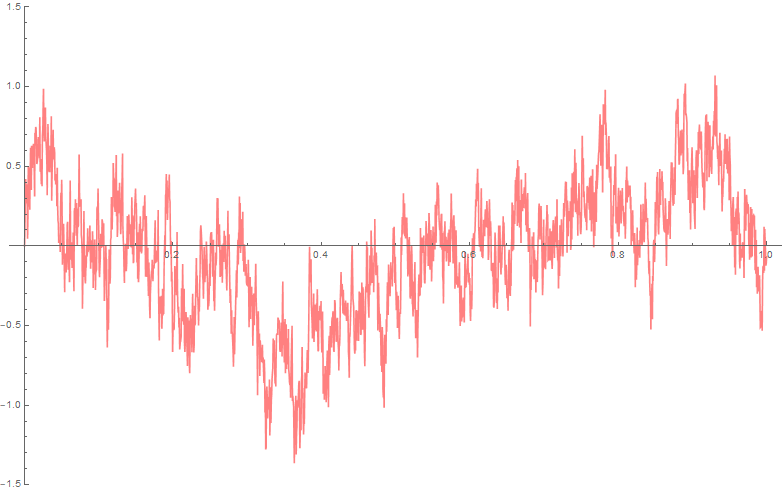}
        \caption{$\gamma = 0.25$ (fBm)}
    \end{subfigure}
	\qquad
    \begin{subfigure}{0.45\textwidth}
        \includegraphics[width=\textwidth]{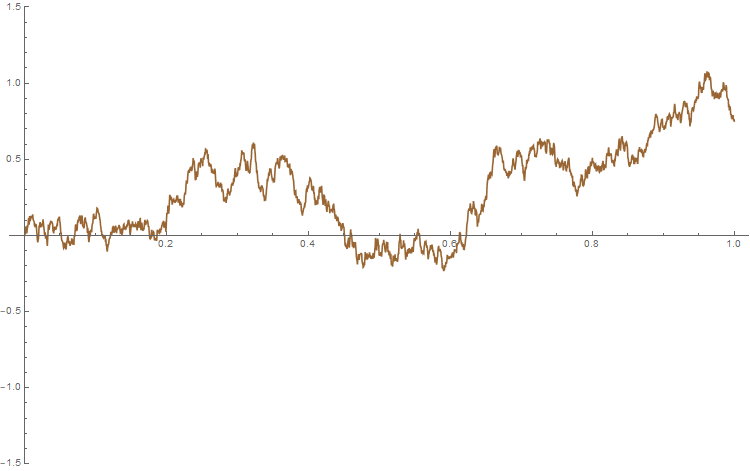}
        \caption{$\gamma = 0.5$ (Brownian motion)}
    \end{subfigure}
	\caption{A fractional Brownian motion and Brownian motion respectively, randomly generated using Mathematica, to be compared with the previous figure.}
\end{figure}

Our proof of this result follows similar ideas as the proof of Theorem \ref{thm:main b}.

Note that $\alpha/2 < 1/2$, so only a fractional Brownian motion with negatively correlated increments can be induced this way. It would be very interesting to understand functional limit theorems of this sort in the context of ergodic processes related to the $B$-frees described in e.g. \cite{Avdeeva2016, Dymek2017, Dymek2018, ElAbdalaoui2015, Kasjan2019, Kulaga2015, Mentzen2017}. There seem to exist only a few other constructions in the literature of a fractional Brownian motion as the limit of a discrete model, e.g. \cite{Arratia,Enqriquez, Hammond, Parczewski, Sottinen}.

\subsection{Notation and conventions}

Throughout the rest of this paper we allow the implicit constants in $\ll$ and $O(\cdot)$ to depend on $k$ when considering a $k$-th moment, and implicit constants are always for a fixed sieving set $B$. Later we will introduce a weight function $\vp$ and implicit constants depend on $\vp$ as well. Throughout the paper where $B$ has an index $\alpha$, for simplicity we will assume that $\alpha \in (0,1)$. Some proofs would remain correct if $\alpha = 0$ or $1$, but the proofs of our central results would not be. We use the notation $a \sim x$ as a subscript in some sums to mean $x < a \leq 2x$. In general we follow standard conventions; in particular $e(x) = e^{i 2\pi x}$, $\|x\|$ denotes the distance of $x \in \R$ to the nearest integer, $(a,b,c)$ is the greatest comment divisor of $a$, $b$, and $c$, while $[a,b,c]$ is the least common multiple.

\subsection{The structure of the proof}

There is a heuristic way to understand the Gaussian variation of $\Nbfree(n,H)$. Note that
\begin{align*}
\Nbfree(n,H) = \sum_{n < vd \leq n+H} \mobb(d) &= \sum_{\ss{d\in \sgbsf \\ d \leq X+H}}\mobb(d)\Big( \Big\lfloor\frac{n+H}{d}\Big\rfloor - \Big\lfloor\frac{n}{d}\Big\rfloor \Big) \\
&= \sum_{\ss{d\in \sgbsf \\ d \leq X+H}} \Big( \frac{\mobb(d)}{d} H + \mobb(d) \Big(\Big\{ \frac{n}{d} \Big\} - \Big\{ \frac{n+H}{d} \Big\} \Big) \Big).
\end{align*}
The contribution of the first summand is close to the value $\meanb H$ around which $\Nbfree(n,H)$ oscillates. On the other hand, the functions $n \mapsto \{(n+H)/d\} - \{n/d\}$ are mean-zero functions of period $d$ and thus are linear combinations of terms $e(n\ell/d)$ for $1\leq \ell \leq d-1$. Upon reducing the fractions $\ell/d$ by the maximal divisor $b | (\ell,d)$ with $b \in \sgb$, we see that $\Nbfree(n,H) - \meanb H$ is approximated by a linear combination of terms $e(n\ell/d)$ for $(\ell,d)$ $B$-free and $d \in \sgbsf$.

Heuristically, if $X$ is large and $n \in [1,X]$ is chosen uniformly at random, one may expect the terms $\{ e(n\ell/d): \; (\ell,d) \text{ $B$-free}\}$ to behave like a collection of independent random variables, and this would imply the Gaussian oscillation of $\Nbfree(n,H)$.

Nonetheless we do not quite have independence; instead, roughly speaking, one may use the same Fourier decomposition to relate the $k$-th moment $M_{k,B}(X,H)$ to (weighted) counts of solutions to the equation
\begin{equation}\label{eq:fundamental_equation}
\frac{\ell_1}{d_1} + \cdots + \frac{\ell_k}{d_k} \equiv 0 \bmod 1,
\end{equation}
where $\|\ell_i/d_i\| < 1/H$ for all $i$. Indeed, the realization that $M_k(X,H)$ for the squarefrees are related to these counts appears already in \cite{Hall2}.

It can be seen that Gaussian behavior will then follow from most solutions to the above equation being diagonal, meaning there is some pairing $\ell_i/d_i \equiv - \ell_j/d_j \bmod 1$ for all $i$, and this is what we demonstrate. Our main tool is the Fundamental Lemma and its extensions (see Lemmas \ref{lem:flgen}, \ref{lem:msgen}, and \ref{lem:lem7gen}), developed by Montgomery--Vaughan \cite{MV} and later used by Montgomery--Soundararajan \cite{MS} to prove a conditional central limit theorem for primes in short intervals.

However, this strategy if used by itself is not sufficient to prove a central limit theorem; the Fundamental Lemma was already known to Hall who used it to obtain his upper bound $M_k(X;H) \ll H^{(k-1)/2}$. The reason this strategy does not work as it did for Montgomery--Soundararajan is that $M_{k,B}(X;H)$ has size roughly $H^{\alpha k/2}$; for primes the $k$-th moment of a short interval count is (conditionally) much larger. Thus in order to recover a main term, our error terms must be shown to be substantially smaller than for the primes.

We obtain an upper bound for the number of off-diagonal solutions to \eqref{eq:fundamental_equation} by bringing in two ideas in addition to those used by Montgomery--Soundararajan. The first is due to Nunes, who showed in the recent paper \cite{Nunes1} that solutions to \eqref{eq:fundamental_equation} make a contribution to $M_k(X;H)$ only when $d_i$ is larger than $H^{1-o(1)}$ for all $i$ and the least common multiple of the $d_i$ is larger than $H^{k/2-o(1)}$; this argument appears in Section \ref{sec:Nunes_bound}. The second is an observation that bounding a term $T_4$ which appears in a variant of the Fundamental Lemma (Lemma \ref{lem:lem7gen}) requires a more delicate treatment than that which appears in \cite[Lemma 8]{MV}.  However, it can be accomplished in our context by counting solutions to a certain congruence equation, which in turn can be estimated using an averaging argument relying crucially on the P\'{o}lya--Vinogradov inequality; this is done in Section \ref{sec:fl_bound}. See Remark \ref{rem:alphaHalf} for further discussion of how the more direct, but less quantitatively precise, approach of \cite[Lemma 8]{MV} is insufficient in our context.

In order to prove that counts of $B$-frees scale to a fractional Brownian motion it is not sufficient to consider the flat counts $\Nbfree(n,H)$, but instead we must consider the weighted counts
\[ \sum_{u \in \Z} \vp\Big(\frac{u-n}{H}\Big) \bfree(u)\]
where $\vp$ belongs to a class of functions that includes step-functions. A proof of a central limit theorem for flat counts remains essentially unchanged as long as $\vp$ is of bounded variation and compactly supported in $[0,\infty)$. Convergence to a fractional Brownian motion as a corollary of this central limit theorem is discussed in Section \ref{sec:fBm_proof}.

\subsection{Acknowledgements}
O.G. is supported by funding from the European Research Council (ERC) under the European Union's Horizon 2020 research and innovation programme (grant agreement No 851318).
Most of this work was completed while A.M. was a CRM-ISM postdoctoral fellow at the Centre de Recherches Math\'{e}matiques. He would like to thank the CRM for its financial support. B.R. received partial support from an NSERC grant and US NSF FRG grant 1854398. Some work on this project was done during a research visit to the American Institute of Mathematics and we thank that institute for its hospitality. We thank Francesco Cellarosi for useful discussions, as well as the anonymous referee for carefully reading the paper and providing a number helpful comments leading to improvements in its exposition.

\section{An expression for moments}

In this section we prove Proposition \ref{prop:limiting_moments}, giving an expression for $C_{k,B}(H)$. The results proved in this section all suppose that the sieving set $B$ is such that $\sgb$ has index $\alpha$. (We do not yet need to suppose that $\sgb$ is regularly varying.)

In order to eventually have a nice framework to prove Theorem \ref{thm:fBm_main} on fractional Brownian motion, we generalize the moments we consider. Suppose $\vp\colon \mathbb{R} \to \mathbb{R}$ is a bounded function supported in a compact subset of $[0,\infty)$. Define $M_{k,B}(X,H;\vp) = M_k(X,H;\vp)$ by
\[ M_k(X,H;\vp) = \frac{1}{X}\sum_{n\le X} \left( \sum_{u \in \Z} \vp\Big(\frac{u-n}{H}\Big) \bfree(u) - \meanb \sum_{h \in \Z} \vp\left(\frac{h}{H}\right) \right)^k.\]
For $\vp = \intervalind$ this recovers $M_{k,B}(X,H)$ as defined above. We have tried to write this paper so that a reader interested only in the more traditional Theorem \ref{thm:bdistGauss} can read it with this specialization in mind.

We introduce the arithmetic function $g_B=g$ defined by 
\[g(n) = \frac{\mobb(n)}{n}\prod_{b\nmid n}(1-b^{-1}).\]

Given a bounded function $\vp\colon \mathbb{R} \to \mathbb{R}$ supported on a compact subset $[0,\infty)$ we also define the $1$-periodic function
\begin{equation}\label{eq:Phi_H_def}
\Phi_H(t) = \sum_{m \in \Z} e(mt) \vp\left(\frac{m}{H}\right),
\end{equation}
where as usual, $e(u) := e^{2\pi i u}$ for $u \in \mb{R}$. 
Finally, given $r \in \sgbsf$ we denote by $\mc{R}_B(r)$ the subset of the group $\R/\Z$ given by
\[ \mc{R}_B(r) = \left\{ \frac{a}{r} : 1 \le a \le r, \, (a,r) \text{ is $B$-free} \right\}/\Z.\]

\begin{prop}\label{prop:general_limiting_moments}
If the sequence $\sgb$ has index $\alpha \in (0,1)$ and $\vp$ is of bounded variation and supported in a compact subset of $[0,\infty)$, then for all fixed integers $k \geq 1$ and any $\e > 0$, as long as $H \leq X$
\[M_k(X,H;\vp) = C_k(H;\vp) + O(H^k X^{\tfrac{\alpha-1}{\alpha+1}+\e} + 1),\]
where $C_k(H;\vp)$ is defined by the absolutely convergent sum
\[C_k(H;\vp) = C_{k,B}(H;\vp) := \sum_{\ss{r_1,\ldots,r_k > 1 \\ r_j \in \sgbsf \,\forall j}} \prod_{j=1}^k g(r_j) \sum_{\ss{ \sg_1 + \cdots + \sg_k \in \Z \\ \sg_i \in \mc{R}_B(r_i) \, \forall i}} \prod_{1 \leq j \leq k} \Phi_H(\sigma_j).\]
\end{prop}

Obviously Proposition \ref{prop:general_limiting_moments} implies Proposition \ref{prop:limiting_moments} by setting $\vp = \intervalind$.

Our ultimate goal will be to prove generalizations of Theorems \ref{thm:GaussianMoments} and \ref{thm:main b}, showing that the moments $M_k(X,H;\vp)$ and $C_k(H;\vp)$ have Gaussian asymptotics for general $\vp$. This will be the content of Theorems \ref{thm:C_general_gaussian} and \ref{thm:CLT_general_weights}.

\subsection{The Fundamental Lemma and other preliminaries}

We prove Proposition \ref{prop:general_limiting_moments} in the next subsection but first we must introduce a few tools.

For $t \in \mb{R}$ and $Y \geq 1$ we let
\begin{equation}\label{eq:EHDef}
E_Y(t) = \sum_{n=1}^Y e(nt).
\end{equation}
Note that $E_H$ is $\Phi_H(t)$ for $\vp = \intervalind$. One has $E_H(t) \ll F_H(t)$ for all $t$, where
\begin{equation}\label{eq:FHbound}
F_H(t) :=\min\left\{H, \frac{1}{\| t \|}\right\},
\end{equation}
and recall $\|t\|$ is the distance from $t$ to the nearest integer. The next lemma studies the first and second moment of $F_H$. The second moment was studied in \cite[Lemma~4]{MV} and the first moment is implicit in \cite[Lemma~2.3]{Nunes1}; we include a proof for completeness.
\begin{lem}\label{lem:FHmoments}
We have
\begin{align} \sum_{\ell=1}^{d-1} F_H\left( \frac{\ell}{d}\right) &\ll d \min \{\log d, \log H\},\\
\sum_{\ell=1}^{d-1} F^2_H\left( \frac{\ell}{d}\right) &\ll d \min \{ d, H\}.
\end{align}
\end{lem}
\begin{proof}
We use \eqref{eq:FHbound} to obtain 
\[ \sum_{\ell=1}^{d-1} F^k_H\left( \frac{\ell}{d}\right) \le \sum_{\ell=1}^{d-1} \min\left\{ H, \left\| \frac{\ell}{d} \right\|^{-1}\right\}^k \le 2\sum_{\max\{1,d/H\} \le j \le d/2}  \frac{d^k}{j^k} + 2\sum_{1 \le j \le d/H} H^k\]
for $k \in \{1,2\}$. If $k=1$, this is 
\[ \le 2d\left( 1+\sum_{\max\{1,d/H\} \le j \le d/2}  \frac{1}{j}\right)\]
and we use the estimate \[\sum_{j \le n} \frac{1}{j} = \log n+O(1)\]
to conclude. For $k=2$, we use \[\sum_{j>n} \frac{1}{j^2} \ll \frac{1}{n}\]
and consider  $d \le H$ and $d>H$ separately.
\end{proof}
\begin{lem}\label{lem:bdPhibyF}
For $\vp$ supported in a compact subset of $[0,\infty)$ and of bounded variation,
\[ \Phi_H(t) \ll F_H(t),\]
where the implied constant depends on $\vp$ only.
\end{lem}

\begin{proof}
Suppose $\vp$ is supported on an interval $[0,c]$. By partial summation,
\[ |\Phi_H(t)| = \Big| \sum_{1 \leq j \leq cH} E_j(t)\Big( \vp\Big(\frac{j-1}{H}\Big) - \vp\Big(\frac{j}{H}\Big)\Big)\Big| \ll F_H(t) \sum_{1 \leq j \leq cH} \Big| \vp\Big(\frac{j-1}{H}\Big) - \vp\Big(\frac{j}{H}\Big) \Big|.\]
As $\vp$ is of bounded variation the claim follows.
\end{proof}

We introduce the arithmetic function
\[\tau_{B,\lcm,k}(d) := |\{d_1,\ldots,d_k \in B : [d_1,\ldots,d_k] = d\}|,\]
\begin{lem}\label{lem:tau_lem_bound}
For a sieving set $B$, for any $k \geq 1$ and any $\e > 0$,
\[\tau_{B,\lcm,k}(d) \ll_{\e} d^\e.\]
\end{lem}

\begin{proof}
Note that $\tau_{B,\lcm,k}(d)$ vanishes for $d \notin \sgb$ and is multiplicative for $d \in \sgb$. If $b \in B$, then $\tau_{B,\lcm,k}(b^r) \ll r^{O_k(1)}$. But for any $\e > 0$ this implies $\tau_{B,\lcm,k}(b^r) \ll A\cdot (b^r)^\e$ for $A$ a constant depending on $\e$ and $k$ only. Thus for $d \in \sgb$, \[\tau_{B,\lcm,k}(d) \ll A^{\omega(d)} d^\e \ll d^{\e+o(1)},\] where $\omega(d)$ is the number of prime factors of $d$ and we use the estimate $\omega(d) = o(\log d)$ (see e.g. \cite[Theorem 2.10]{MNTI}). This implies the claim.
\end{proof}

The next lemma is a variation on an estimate of Nunes \cite[Lemma 2.4]{Nunes1}, who treated the corresponding result when $B = \{p^m : p \text{ prime}\}$ with $m \geq 2$. 

\begin{lem}\label{lem:corrlarged}
Suppose the sequence $\sgb$ has index $\alpha \in (0,1)$. Let $x>1$ and $m_1,\ldots,m_k$ be a $k$-tuple of distinct non-negative integers and let $X=\max_{1 \le i \le k}(x+m_i)$. We have
\begin{equation}\label{eq:largeprod} \sum_{\substack{d_1,\ldots,d_k \in \sgbsf\\ [d_1,\ldots,d_k] > z }} \; \sum_{n\leq x} \prod_{i=1}^k \Big( \mathbf{1}_{n \equiv - m_i \; (d_i)} + \frac{1}{d_i}\Big) \ll X^{o(1)} \left( X z^{\alpha-1+o(1)} + X^{\frac{2\alpha}{\alpha+1}} \right).
\end{equation}
\end{lem}

\begin{proof}
We prove the claim by induction on $k$. For $k=1$, the inner sum in the left-hand side of \eqref{eq:largeprod} is $\ll X/d_1$ (by considering $X> d_1$ and $X\le d_1$ separately), so we obtain the bound
\[ \sum_{\substack{d_1 \in \sgbsf\\ d_1 > z}} \sum_{n \le x} \left(\mathbf{1}_{n \equiv -m_1 \;(d_1)} + \frac{1}{d_1}\right)\ll
\sum_{\substack{d_1 \in \sgbsf\\ d_1>z}} \frac{X}{d_1} \ll X z^{\alpha-1+o(1)},\] 
which is stronger than what is needed. We now assume that the bound holds for $k_0$ and prove it for $k=k_0+1$.

We set $D= [d_1, d_2, \ldots, d_{k_0+1}]$ and $D_i = [d_1,\ldots,d_{i-1},d_{i+1},\ldots,d_k]$ and introduce a parameter $z' \le X$ to be chosen later. We write the left-hand side of \eqref{eq:largeprod} as $T_1 +T_2$, where in $T_1$ we sum only over tuples  $d_1,\ldots,d_{k_0+1}$ with $D_i\le z'$ holding for all $1 \le i \le k_0 + 1$, and in $T_2$ we sum over the rest. 

Observe that by comparing exponents in the factorizations, $D \leq \prod_{i=1}^{k_0+1} (D_i)^{1/k_0}$. To bound $T_1$, observe that the condition on $D_i$ implies $D  \le (z')^{1+1/k_0}$. Using the Chinese remainder theorem, the inner sum for $T_1$ is $\ll X/D + 1$, and so we obtain the upper bound
\[ T_1 \ll \sum_{\substack{d_1, d_2,\ldots, d_{k_0+1} \in \sgbsf\\ z < D \le (z')^{1+1/k_0}}} \left( \frac{X}{D}+1\right) \le \sum_{\substack{d \in \sgbsf\\ z < d \le (z')^{1+1/k_0}}} \tau_{B,\lcm,k_0+1}(d) \left( \frac{X}{d}+1\right)\ll X^{o(1)} \left( Xz^{\alpha-1+o(1)} + (z')^{2\alpha}\right),\]
where we have used Lemma \ref{lem:tau_lem_bound} in the last step.

For each of the tuples summed over in $T_2$, there is some $j$ ($1 \le j \le k_0 + 1$) with $D_j > z'$. The tuples corresponding to this $j$ contribute to $T_2$ at most
\begin{multline*} \sum_{\substack{d_1,\ldots,d_{j-1},d_{j+1},\ldots,d_{k_0+1} \in \sgbsf \\ D_j > z' }} \sum_{n \leq x} \prod_{i \neq j}\Big( \mathbf{1}_{n \equiv - m_i \; (d_i)} + \frac{1}{d_i}\Big) \sum_{d_j \in \sgbsf} \Big( \mathbf{1}_{n \equiv - m_j \; (d_j)} + \frac{1}{d_j}\Big) \\
\ll X^{o(1)} \sum_{\substack{d_1,\ldots,d_{j-1},d_{j+1},\ldots,d_{k_0+1} \in \sgbsf\\ D_j > z' }} \sum_{n \leq x} \prod_{i \neq j}\Big( \mathbf{1}_{n \equiv - m_i \; (d_i)} + \frac{1}{d_i}\Big),
\end{multline*}
where we used the facts that (i) the number of $d_j \mid n+m_j$ will be $\ll X^{o(1)}$, and (ii) $\sum 1/d_j = O(1)$.

Applying the induction hypothesis to bound the last double sum, we obtain
\[ T_2 \ll X^{o(1)}\left( X (z')^{\alpha-1+o(1)}  + X^{\frac{2\alpha}{\alpha+1}} \right).\]
Taking $z'=X^{1/(\alpha+1)}$, we see that $T_1+T_2$
does not exceed the desired bound. 
\end{proof}

Finally, throughout this paper in order to control the inner sums defining $C_k(H;\vp)$ we will use the Fundamental Lemma of Montgomery and Vaughan \cite{MV}. The following is a generalization of the result proved in \cite{MV}, which corresponds to the case of $B$ being the set of primes. The original proof in \cite{MV} works without any change under the more general assumptions. 
Let 
\[ \mc{C}(r) := \left\{ \frac{a}{r} : 1 \le a \le r \right\} / \Z.\]
\begin{lem}[Montgomery and Vaughan's Fundamental Lemma]\label{lem:flgen}
	Let $r_1$, $r_2$, \ldots, $r_k$ be positive integers from $\sgbsf$, and set $r=[r_1,\ldots,r_k]$. For each $1 \le i \le k$, let $G_i(\rho_i)$ be a complex-valued function defined on $\mc{C}(r_i)$. Suppose each prime factor of $r$ divides at least two of the $r_i$.\footnote{Equivalently, each $b \in B$ that divides $r$ divides at least two of the $r_i$.} Then
	\[ \left| \sum_{\substack{\rho_i \in \mc{C}(r_i)\\\sum_{i=1}^{k} \rho_i\equiv  0 \bmod 1}} \prod_{i=1}^{k} G_i(\rho_i) \right| \le \frac{1}{r}\prod_{i=1}^{k} \left(r_i \sum_{\rho_i \in \mc{C}(r_i)} |G_i(\rho_i)|^2\right)^{\frac{1}{2}}.\]
\end{lem}

Later, in Lemmas \ref{lem:msgen} and \ref{lem:lem7gen}, we will cite variants of this result.

\subsection{Proof of Proposition \ref{prop:general_limiting_moments}}

\begin{proof}
We examine the inner sum defining $M_k(X,H;\vp)$. Note for integers $n$,
\begin{align*}
\sum_{u \in \Z} \vp\Big(\frac{u-n}{H}\Big) \bfree(u) - \meanb \sum_{h \in \Z} \vp\left(\frac{h}{H}\right) &= \sum_{d \in \sgbsf} \sum_{v \in \Z} \vp\Big(\frac{vd - n}{H}\Big) \mobb(d) - \sum_{d \in \sgbsf} \frac{\mobb(d)}{d} \sum_{h \in \Z} \vp\left(\frac{h}{H}\right) \\
&= \sum_{d \in \sgbsf} \mobb(d) \psi_H(n,d),
\end{align*}
where for notational reasons we have written 
\begin{equation} \label{eq:psi_def}
\psi_H(n,d) := \sum_{m \in \Z} \vp\Big(\frac{m}{H}\Big) \Big(\mathbf{1}_{m \equiv - n\; (d)} - \frac{1}{d}\Big).
\end{equation}
The function $n\mapsto \psi_H(n,d)$ has period $d$. Considering it as a function on $\Z/d\Z$ it has mean $0$. By taking the finite Fourier expansion in $n$ we have
\begin{equation}\label{eq:psi_to_Phi}
\psi_H(n,d) = \frac{1}{d} \sum_{\ell=1}^{d-1} \Phi_H\left(\frac{\ell}{d}\right) e\left(\frac{n\ell}{d}\right).
\end{equation}
Thus
\begin{equation}\label{eq:moment_to_psi}
M_k(X, H; \vp) = \frac{1}{X} \sum_{n \leq X} \sum_{d_1, \ldots , d_k \in \sgbsf} \prod_{j=1}^k \mobb(d_j) \psi_H(n,d_j).
\end{equation}
From the definition \eqref{eq:psi_def}, each term $\psi_H(n,d)$ involves summing over $\ll H$ indices $m$. Thus from Lemma \ref{lem:corrlarged} we have
\begin{equation}\label{eq:large_lcd_bound}
\frac{1}{X} \sum_{n \leq X} \sum_{\substack{d_1,\ldots, d_k \in \sgbsf\\ [d_1,\ldots,d_k] > z}} \prod_{j=1}^k \mobb(d_j) \psi_H(n,d_j) \ll \frac{H^k}{X} \cdot X^{o(1)}( X z^{\alpha-1+o(1)} + X^{\frac{2\alpha}{\alpha+1}}),
\end{equation}
for a parameter $z$ to be chosen later.

On the other hand, by \eqref{eq:psi_to_Phi},
\begin{multline}\label{eq:small_lcd_transform} 
\frac{1}{X} \sum_{n \leq X} \sum_{\substack{d_1,\ldots, d_k \in \sgbsf\\ [d_1,\ldots,d_k] \leq z}} \prod_{j=1}^k \mobb(d_j) \psi_H(n,d_j) \\
= \sum_{\substack{d_1,\ldots, d_k \in \sgbsf\\ [d_1,\ldots,d_k] \leq z}} \frac{\mobb(d_1) \cdots \mobb(d_k)}{d_1\cdots d_k} \sum_{\ss{0 < \ell_j < d_j \\ \forall 1 \leq j \leq k}} \Phi_H\Big(\frac{\ell_1}{d_1}\Big)\cdots \Phi_H\Big(\frac{\ell_1}{d_1}\Big) \Big( \frac{1}{X} E_X\Big( \frac{\ell_1}{d_1} + \cdots +\frac{\ell_k}{d_k}\Big)\Big),
\end{multline}
where $E_X$ is defined as in \eqref{eq:EHDef}. Note that if $\ell_1/d_1 + \cdots+ \ell_k/d_k \notin \Z$, we have
\[
\left\| \frac{\ell_1}{d_1} + \cdots +\frac{\ell_k}{d_k} \right\| \geq \frac{1}{[d_1,\ldots,d_k]}
\]
and in this latter case
\[
\frac{1}{X} E_X\Big( \frac{\ell_1}{d_1} + \cdots + \frac{\ell_k}{d_k}\Big) \ll \frac{[d_1,\ldots,d_k]}{X}.
\]
Furthermore, note using Lemma \ref{lem:bdPhibyF} and the first part of Lemma \ref{lem:FHmoments} that
\[
\sum_{0 < \ell < d} \Phi_H\left(\frac{\ell}{d}\right) \ll \sum_{0 < \ell < d} F_H\left(\frac{\ell}{d}\right) \ll d \log H.
\]
Thus the contribution of terms for which $\ell_1/d_1 + \cdots + \ell_k/d_k \notin \Z$ is
\[
\ll H^{o(1)} \sum_{\substack{d_1,\ldots, d_k \in \sgbsf \\ [d_1,\ldots,d_k] \leq z}} \frac{[d_1,\ldots,d_k]}{X} \leq \frac{H^{o(1)}}{X} \sum_{\ss{d \in \sgbsf \\ d \leq z}} \tau_{B,\lcm,k}(d) d = \frac{H^{o(1)} z^{1+\alpha + o(1)}}{X}.
\]
Thus \eqref{eq:small_lcd_transform} is
\begin{equation}\label{eq:psi_to_main}
\sum_{\substack{d_1,\ldots, d_k \in \sgbsf\\ [d_1,\ldots,d_k] \leq z}} \frac{\mobb(d_1) \cdots \mobb(d_k)}{d_1\cdots d_k} \sum_{\ss{0 < \ell_j < d_j \, \forall j \\ \sum \ell_i/d_i \in \Z}} \Phi_H\Big(\frac{\ell_1}{d_1}\Big)\cdots \Phi_H\Big(\frac{\ell_1}{d_1}\Big) + O\left(\frac{H^{o(1)} z^{1+\alpha+o(1)}}{X}\right).
\end{equation}
We now complete the sum above. Directly applying Lemma \ref{lem:flgen} (and appealing to Lemma \ref{lem:bdPhibyF} and the second part of Lemma \ref{lem:FHmoments}), we see that the corresponding sum over tuples \[\mbf{d} = (d_1,\ldots,d_k)\text{ with }[d_1,\ldots,d_k] > z\]
is
\begin{equation}\label{eq:main_completion}
\ll H^{\frac{k}{2}} \sum_{\substack{d_1,\ldots, d_k \in \sgbsf\\ [d_1,\ldots,d_k] > z}} \frac{1}{[d_1,\ldots,d_k]} \ll H^{\frac{k}{2}} z^{\alpha-1+o(1)}.
\end{equation}
Hence from \eqref{eq:large_lcd_bound}, \eqref{eq:psi_to_main}, \eqref{eq:main_completion},
\begin{multline*}
M_k(X,H;\vp) = C_k^\ast(H;\vp)
+ O(H^k z^{\alpha-1+o(1)} X^{o(1)} + H^k X^{\frac{\alpha-1}{\alpha+1} + o(1)} + H^{o(1)} z^{1+\alpha+o(1)} X^{-1}),
\end{multline*}
where
\begin{equation}\label{eq:C_k_alt}
C_k^\ast(H;\vp) = \sum_{d_1,\ldots, d_k \in \sgbsf} \frac{\mobb(d_1) \cdots \mobb(d_k)}{d_1\cdots d_k} \sum_{\ss{0 < \ell_j < d_j \, \forall j \\ \sum \ell_i/d_i \in \Z}} \Phi_H\Big(\frac{\ell_1}{d_1}\Big)\cdots \Phi_H\Big(\frac{\ell_1}{d_1}\Big).
\end{equation}
(Note that the absolute convergence of this sum is implied by the above derivation.) Setting $z = X^{(1-\e)/(\alpha+1)}$ we obtain the desired error term.

It remains to demonstrate $C_k^\ast(H;\vp) = C_k(H;\vp)$. To do this, note that if $\sg_i = \ell_i/d_i$, with $0 < \ell_i < d_i$,  we can find a maximal $e_i \in \sgbsf$ such that $\ell_i = e_i l_i'$, $d_i = e_i d_i'$, and so that $(l_i',d_i')$ is $B$-free. Moreover, writing $\sg_i = l_i'/d_i'$ does not affect the condition $\sg_1 + \cdots + \sg_k \in \mb{Z}$. Consequently, we have
\[
\sum_{\ss{0 < \ell_j < d_j \forall j \\ \sum \ell_i/d_i \in \mb{Z}}} \Phi_H\left(\frac{\ell_1}{d_1}\right) \cdots \Phi_H\left(\frac{\ell_k}{d_k}\right) = \sum_{\substack{e_1\mid d_1\\e_1 \in \sgbsf}} \cdots \sum_{\substack{e_k \mid d_k\\e_k \in \sgbsf}} \sum_{\ss{\sg_1', \ldots, \sg_k' \\ \sg_j' \in \mc{R}_B(d_j/e_j) \, \forall j \\ \sg_1' + \cdots + \sg_k' \in \mb{Z}}} \Phi_H(\sg_1')\cdots \Phi_H(\sg_k').
\]
Note also that as $\ell_i \neq 0$, $d_i/e_i \neq 1$. Therefore, setting $r_j := d_j/e_j$ in each of the above sums, we have
\begin{align*}
C_k^\ast(H;\vp)  = \sum_{\ss{r_1,\ldots,r_k > 1 \\ r_j \in \sgbsf \, \forall j}} \frac{\mobb(r_1)\cdots \mobb(r_k)}{r_1\cdots r_k} \left(\sum_{\ss{e_1,\ldots,e_k \geq 1 \\ e_j \in \sgbsf \, \forall j \\ (e_j,r_j) \text{ is $B$-free}}} \frac{\mobb(e_1)\cdots \mobb(e_k)}{e_1\cdots e_k}\right) \sum_{\ss{\sg_1,\ldots,\sg_k \\ \sg_j \in \mc{R}_B(r_j) \\ \sg_1 + \cdots + \sg_k \in \mb{Z}}} \Phi_H(\sg_1) \cdots \Phi_H(\sg_k).
\end{align*}
The sums over $e_j$ factor as
\[
\prod_{j =1}^k \prod_{\ss{b \nmid r_j \\ b \in B}} \left(1-\frac{1}{b}\right),
\]
and we see $C_k^\ast(H;\vp) = C_k(H;\vp)$ as required.
\end{proof}

\begin{rem}\label{rem:C_k_alt}
From \eqref{eq:C_k_alt}, reindexing $\ell_1' = \ell_1$ and $\ell_2' = d_2 - \ell_2$, we have the following useful alternative expression for $C_2(H;\vp)$:
\begin{equation}\label{eq:C_2_complete}
C_2(H;\vp) = \sum_{d_1, d_2 \in \sgbsf} \frac{\mobb(d_1)\mobb(d_2)}{d_1 d_2} \sum_{\ss{0 < \ell_1' < d_1 \\ 0 < \ell_2' < d_2 \\ \ell_1'/d_1 = \ell_2'/d_2 }} \Phi_H\Big(\frac{\ell_1'}{d_1}\Big)\Phi_H\Big(-\frac{\ell_2'}{d_2}\Big).
\end{equation}
The absolute convergence of this sum is implied by the above argument. 
\end{rem}

\begin{rem}
When $\vp = \intervalind$ the corresponding expression for $C_2(H;\vp)$ may also be calculated using correlation formulae for $\mobb$. Indeed, upon expanding the $k$-th power in the definition of $M_k(X,H;\intervalind$ and swapping orders of summation, we find
\[
M_k(X,H;\intervalind) = \sum_{0 \leq j \leq k} \binom{k}{j} (- \mc{M}_BH)^{k-j} \sum_{0 \leq h_1,\ldots,h_j \leq H} \frac{1}{X} \sum_{n \leq X} \bfree(n+h_1)\cdots \bfree(n+h_j).
\]
When $B = \{p^2 : p \text{ prime}\}$, Hall \cite[Lemma~2]{Hall2} used this approach to compute the main terms $C_k(H; \intervalind)$. This was generalized straightforwardly to the setting with $B = \{p^m : p \text{ prime}\}$ for $m \geq 3$ by Nunes \cite[Lemma~2.2]{Nunes1} and can be generalized to $B$-frees as well. 
\end{rem}

\section{Estimates for variance}

\subsection{Preliminary results on the index of $B$, $\sgbsf$, and $\sgb$}

In this section we will estimate the variance $C_2(H)$ in various ways. We first prove some preliminary results relating the index of $\sgb$ to $\sgbsf$ and $B$ itself.

\begin{lem}\label{lem:index_sgb_to_sgbsf}
For a sieving set $B$, the sequence $\sgb$ has index $\alpha$ if and only if $\sgbsf$ has index $\alpha.$
\end{lem}

\begin{proof}
Suppose $\sgb$ has index $\alpha$. We first show that $\sgbsf$ has index $\alpha$ also. Plainly,\[N_{\sgbsf}(x) \leq N_{\sgb}(x) = x^{\alpha+o(1)}.\] On the other hand,
\begin{equation}\label{eq:N_sgb_to_N_sgbsf}
\sum_{ b \leq x, b \in \sgb} N_{\sgbsf}\left(\frac{x}{b^2}\right) = N_{\sgb}(x).
\end{equation}
Introducing a parameter $y = x^\e$, we upper-bound the left-hand sum as follows. For $b \geq y$ with $b \in \sgb$ we apply $N_{\sgbsf}(x/b^2) \leq N_{\sgb}(x/b^2)$, and for $b < y$ we use monotonicity in the form $N_{\sgbsf}(x/b^2) \leq N_{\sgbsf}(x)$. Thus, 
\[
N_{\sgb}(y) N_{\sgbsf}(x) \geq N_{\sgb}(x) - \sum_{b \geq y, b \in \sgb} N_{\sgb}\left(\frac{x}{b^2}\right) \geq x^{\alpha+o(1)},
\]
so $N_{\sgbsf}(x) \geq x^{(1-\e)\alpha + o_\e(1)}$. Since this holds for every $\e > 0$, we obtain the lower bound needed for the claim.

In the converse direction, if $\sgbsf$ has index $\alpha$, then \[N_{\sgb}(x) \geq N_{\sgbsf}(x) = x^{\alpha+o(1)}.\] To prove an upper bound for $N_{\sgb}(x)$, use \eqref{eq:N_sgb_to_N_sgbsf} and note the left-hand side will be $\ll x^{\alpha+o(1)}$ since $\sum_{b \in \sgb} b^{-2\alpha} < \infty$.
\end{proof}

Other authors have proved results in this area based on assumptions regarding the index of the set $B$ (for instance \cite{Grimmett1997}). Though we will not require it in what follows, for completeness' sake we note the following implication.

\begin{prop}\label{prop:index_b_to_sgb}
For a sieving set $B$, if $B$ has index $\alpha$, then $\sgb$ also has index $\alpha.$
\end{prop}

\begin{proof}
As $B \subseteq \sgb$, it suffices to prove an upper bound on $N_{\sgb}(x)$. For $n \in \sgb$, let $\omega_B(n):=\#\{b \in B: b \mid n\}$ and define \[N_{\sgbsf}^{(k)}:= \{n \leq x:\, n \in \sgbsf,\, \omega_B(n) \leq k\}.\] We will first show for any $\e > 0$ there is a constant $C_\e$ such that
\begin{equation}\label{eq:N_sgb_k_bound}
N_{\sgbsf}^{(k)}(x) \leq C_\e^k x^{\alpha+\e}.
\end{equation}
For notational reasons let $B^\dagger:= B \cup \{1\}$. Fix $\e > 0$. It is plain there exists a constant $C_\e$ such that
\[
\sum_{b\leq x,\; b\in B^\dagger} \frac{1}{b^{\alpha+\e}} \le C_\e, \quad \text{and} \quad N_{B^\dagger}(x) \leq C_\e x^{\alpha+\e}.
\]
This gives \eqref{eq:N_sgb_k_bound} for $k=1$. But \eqref{eq:N_sgb_k_bound} then follows inductively for all $k$ from the above bounds and
\[
N_{\sgbsf}^{(k)}(x) \le \sum_{b \leq x, \, b \in B^\dagger} N_{\sgbsf}^{(k-1)}\left(\frac{x}{b}\right).
\]
Now note $\omega_B(n) \leq \omega(n) = o(\log x)$ for all $n \leq x$ (see \cite[Theorem 2.10]{MNTI}). Hence for all $x$ there exists $k = o(\log x)$ such that
\[
N_{\sgbsf}(x) = N_{\sgbsf}^{(k)}(x) \leq C_\e^k x^{\alpha+\e} = x^{\alpha+\e + o(1)}.
\]
As $\e$ is arbitrary this implies \[N_{\sgbsf}(x) \ll x^{\alpha+o(1)}.\] 
Using \eqref{eq:N_sgb_to_N_sgbsf} as before this implies $N_{\sgb}(x) \ll x^{\alpha+o(1)}$ as desired.
\end{proof}

It seems likely that the converse to Proposition \ref{prop:index_b_to_sgb} is false, but we do not pursue this here.

\subsection{Variance for $\sgb$ with index $\alpha$}

We now show that the exponent of the variance is determined by the index of $\sgb$.
\begin{lem}\label{lem:estimate c2rho}
Suppose $\vp$ is of bounded variation, supported on a compact subset of $[0,\infty)$. Assume moreover that $\vp$ is non-vanishing on some open interval. If $\sgb$ is of index $\alpha \in (0,1)$, then 
\[
C_2(H;\vp) = H^{\alpha+o(1)}.
\]
\end{lem}

\begin{proof}
We have
\[
C_2(H;\vp) = \sum_{\ss{ r > 1 \\ r \in \sgbsf}} g(r)^2 \sum_{\sg \in \mc{R}_B(r)} |\Phi_H(\sg)|^2.
\]
For an upper bound we apply Lemma \ref{lem:bdPhibyF} and the second part of Lemma \ref{lem:FHmoments} to see that
\begin{equation}\label{eq:complete_sum_F}
\sum_{\sg \in \mc{R}_B(r)} |\Phi_H(\sg)|^2 \ll r\min\{r,H\}.
\end{equation}
As $g(r) \ll 1/r$ we have
\[
C_2(H;\vp) \ll \sum_{\ss{ r > 1 \\ r \in \sgbsf}} \frac{1}{r^2} \cdot r \min\{r,H\} \ll H^{\alpha+o(1)}.
\]
Suppose $\vp\colon \mathbb{R} \to \mathbb{R}$ is supported on an interval $[0,c]$. Before embarking on the proof of the lower bound, we make the following observation. Since $\vp$ is of bounded variation we have
\begin{align}\label{eq:SumtoIntvp}
\left|\Phi_H(0) - H\int_0^{\infty} \vp(t)dt\right| &\leq \sum_{0 \leq m \leq cH} \left|\vp\left(\frac{m}{H}\right) - H\int_{m/H}^{(m+1)/H} \vp(t)dt\right|  \\
&\leq H\int_0^{1/H} \sum_{0 \leq m \leq cH} \left|\vp\left(\frac{m}{H}\right)-\vp\left(\frac{m}{H}+u\right)\right| du \ll 1.
\end{align}
We now split the proof of the lower bound into two cases, depending on whether or not $\int_0^{\infty} \vp(t) dt = 0$. 
\newline
\newline
\noindent \textit{Case 1: $\int_0^{\infty} \vp(t) dt = 0$}. From \eqref{eq:SumtoIntvp} we obtain $\Phi_H(0) = O(1)$. Now, for convenience set $\nu_H(r) = 0$ if $r \notin \sgbsf$, and otherwise put
\[
\nu_H(r) := \sum_{\sg \in \mc{R}_B(r)} |\Phi_H(\sg)|^2 =
\sum_{\ss{a \bmod{r} \\ (a,r) \text{ is $B$-free}}} \left|\Phi_H\left(\frac{a}{r}\right)\right|^2.
\]
Using the convolution formula $\bfree= \mathbf{1} \ast \mobb$, for each $r \in \sgbsf$ we obtain
\[
\nu_H(r) = \sum_{a \bmod{r}} (\mathbf{1} \ast \mobb)((a,r)) \left|\Phi_H\left(\frac{a}{r}\right)\right|^2 = \sum_{\ss{md = r \\ m,d \in \sgbsf}} \mobb(d) \sum_{b \bmod{m}} \left|\Phi_H\left(\frac{b}{m}\right)\right|^2 = \sum_{\ss{md = r \\ m,d \in \sgbsf}} \mobb(d) \sum_{\sg \in \mc{C}(m)} |\Phi_H(\sg)|^2.
\]
Moreover, if $m \geq 10cH$ then by Plancherel's theorem on $\mb{Z}/m\mb{Z}$, we obtain
\begin{equation}\label{eq:L2modm}
\frac{1}{m}\sum_{\sg \in \mc{C}(m)}|\Phi_H(\sg)|^2 = \sum_{\ss{n_1,n_2 \in \mb{Z} \\ m\mid (n_1-n_2)}} \vp\left(\frac{n_1}{H}\right) \vp\left(\frac{n_2}{H}\right) = \sum_{n \in \mb{Z}} \vp\left(\frac{n}{H}\right)^2 \gg H,
\end{equation}
where in the last step we used the fact that $\vp$ is non-vanishing on some open interval. 

Since $g(r)^2 \asymp r^{-2}$ for $r \in \sgbsf$, we find that
\begin{align*}
C_2(H;\vp) &\gg \sum_{\ss{r > 1}} \frac{\nu_H(r)}{r^2} = \sum_{\ss{md > 1 \\ m,d \in \sgbsf, (d,m) = 1}} \frac{\mobb(d)}{d^2} \cdot \frac{1}{m^2} \sum_{\sg \in \mc{C}(m)}|\Phi_H(\sg)|^2 \\
&= \sum_{\ss{m \geq 1 \\ m \in \sgbsf}} \frac{1}{m^2} \sum_{\sg \in \mc{C}(m)} |\Phi_H(\sg)|^2 \sum_{\ss{d \geq 1 \\ (d,m) = 1}} \frac{\mobb(d)}{d^2} - |\Phi_H(0)|^2 \\
&= \sum_{\ss{m \geq 1 \\ m \in \sgbsf}} \frac{1}{m^2} \sum_{\sg \in \mc{C}(m)} |\Phi_H(\sg)|^2 \prod_{\ss{b \in B \\ b \nmid m}}(1-b^{-2}) - O(1).
\end{align*}
Since \[\prod_{\ss{b \in B \\ b \nmid m}} (1-b^{-2}) \geq \prod_{b \in B} (1-b^{-2}) > 0\] uniformly in $m$, we may use positivity to restrict to $m \geq 10 cH$ and apply \eqref{eq:L2modm}, getting
\[
C_2(H;\vp) \gg \sum_{\ss{m \geq 10cH \\ m \in \sgbsf}} \frac{1}{m} \left(\frac{1}{m} \sum_{\sg \in \mc{C}(m)} |\Phi_H(\sg)|^2\right) \gg H\sum_{\ss{m \geq 10cH \\ m \in \sgbsf}} \frac{1}{m} \gg H \cdot H^{\alpha-1+o(1)} = H^{\alpha+o(1)},
\]
where in the second to last step we used the fact that $\sgbsf$ has index $\alpha$ by Lemma \ref{lem:index_sgb_to_sgbsf}. This proves the lower bound in this case. 
\newline
\newline
\noindent \textit{Case 2: $\int_0^{\infty} \vp(t)dt \neq 0$}. In this case, we again use positivity to restrict the sum in $C_2(H;\vp)$ as
\[
C_2(H;\vp) \geq \sum_{\ss{ r > 1 \\ r \in \sgbsf}} g(r)^2 \left|\Phi_H\left(\frac{1}{r}\right)\right|^2.
\]

Let $K \geq 10$ be a large constant. Then for $r \geq KcH$, we have $e(j/r) = 1 + O(1/K)$ uniformly for $1\leq j \leq cH$, so from \eqref{eq:SumtoIntvp} we get
\begin{align*}
\left|\Phi_H\left(\frac{1}{r}\right)\right| &\geq \Big| \sum_{1\leq j \leq cH} \vp\left(\frac{j}{H}\right)\Big| - O\left(\frac{1}{K} \sum_{1 \leq j \leq cH} \left|\vp\left(\frac{j}{H}\right)\right|\right) = |\Phi_H(0)| - O\left(\frac{H}{K}\right) \\
&\geq H\left(\left|\int_0^{\infty} \vp(t) dt\right| - O\left(\frac{1}{K}\right)\right) - O(1) \gg H,
\end{align*}
if $K$ is large enough.

Since $g(r)^2 \gg r^{-2}$ for $r \in \sgbsf$, we have
\[
C_2(H;\vp) \gg H^2 \sum_{\ss{ r > K cH \\ r \in \sgbsf}} \frac{1}{r^2} \gg H^{\alpha+o(1)},
\]
since, again by Lemma \ref{lem:index_sgb_to_sgbsf}, $\sgbsf$ has index $\alpha$. The claim is thus proved in this case as well.
\end{proof}

Obviously this implies Proposition \ref{prop:index_variance} where $\vp = \intervalind$.

\subsection{Variance for regularly varying $\sgb$}

If $\sgb$ is a regularly varying sequence we can say more; in this case we will show the asymptotic formula of Proposition \ref{prop:regularlyvarying_variance}.

We begin with a useful expression for $C_2(H)$. Throughout this subsection we use the notation
\[
V(t) = \frac{\sin \pi t}{\pi t}.
\]
\begin{lem}\label{lem:continous_variance_sum}
We have
\begin{equation}\label{eq:C_2_sinc}
C_2(H) = 2H^2 \sum_{d \in \sgbsf} \frac{1}{d^2} \prod_{\substack{b \nmid d \\ b \in B}} \Big(1 - \frac{2}{b}\Big) \sum_{\lambda\geq 1} \left|V\left(\frac{H\lambda}{d}\right)\right|^2.
\end{equation}
\end{lem}

\begin{proof}
We begin with the expression in Remark \ref{rem:C_k_alt} for $C_2(H;\vp)$. If we specialize to the case $\vp = \intervalind$, then $\Phi_H(t) = E_H(t)$. If we use the identity \[E_H(t) = e\left(\frac{H+1}{2} t\right) \frac{\sin( \pi H t)}{\sin (\pi t)},\] we see that \eqref{eq:C_k_alt} gives
\[
C_2(H) = \sum_{d_1, d_2 \in \sgbsf} \frac{\mobb(d_1)\mobb(d_2)}{d_1 d_2} \sum_{\ss{0 < \ell_1' < d_1 \\ 0 < \ell_2' < d_2 \\ \ell_1'/d_1 = \ell_2'/d_2 }} \Big(\frac{\sin(\pi H \ell_1'/d_1)}{\sin(\pi \ell_1'/d_1)}\Big) \Big(\frac{\sin(\pi H \ell_2'/d_2)}{\sin(\pi \ell_2'/d_2)}\Big).
\]
But in this sum $\ell_1'/d_1 = \ell_2'/d_2$ and using
\[
\frac{1}{(\sin \pi x)^2} = \sum_{k \in \Z} \frac{1}{(\pi(x+k))^2},
\]
we obtain
\begin{multline*}
C_2(H) = H^2 \sum_{d_1, d_2 \in \sgbsf} \frac{\mobb(d_1)\mobb(d_2)}{d_1 d_2} \sum_{\ss{\lambda_1, \lambda_2 \neq 0 \\ \lambda_1/d_1 = \lambda_2/d_2 }} V\left(\frac{H \lambda_1}{d_1}\right) V\left(\frac{H \lambda_2}{d_2}\right) \\
= 2H^2 \sum_{d_1, d_2 \in \sgbsf} \frac{\mobb(d_1)\mobb(d_2)}{d_1 d_2} \sum_{\ss{\lambda_1, \lambda_2 \ge 1 \\ \lambda_1/d_1 = \lambda_2/d_2 }} V\left(\frac{H \lambda_1}{d_1}\right) V\left(\frac{H \lambda_2}{d_2}\right).
\end{multline*}
Write $d = (d_1,d_2)$ and $d_1 = \nu_1 d$, $d_2 = \nu_2 d$ and parameterize solutions to $\lambda_1/d_1 = \lambda_2/d_2$ by $\lambda_1 = \lambda \nu_1$, $\lambda_2 = \lambda \nu_2$. Then the above expression for $C_2(H)$ simplifies to
\[
C_2(H) = 2 H^2 \sum_{d \in \sgbsf} \sum_{\ss{\nu_1, \nu_2 \in \sgbsf \\ (\nu_1, \nu_2) = 1}} \frac{\mobb(d \nu_1) \mobb(d \nu_2)}{d^2 \nu_1 \nu_2} \sum_{\lambda > 0} V\left( \frac{H\lambda}{d}\right)^2,
\]
which in turn simplifies to \eqref{eq:C_2_sinc}.
\end{proof}

We will use the following result of P\'{o}lya to estimate the above sum.

\begin{prop}[P\'{o}lya]\label{prop:regular_variation_asymp}
If $N_R(x)$ is the counting function of a sequence $R$ which regularly varies with index $\alpha$, and if $f\colon [0,\infty) \rightarrow \R$ is a function that is Riemann integrable over every finite interval $[a,b] \subset [0,\infty)$ which satisfies
\[
|f(x)| \ll \begin{cases} x^{- \alpha + \varepsilon} & \textrm{as}\; x \rightarrow 0, \\ x^{-\alpha - \varepsilon} & \textrm{as}\; x \rightarrow \infty,\end{cases}
\]
for some $\varepsilon > 0$, then 
\[
\lim_{X\rightarrow\infty} \frac{1}{N_R(X)} \sum_{r \in R} f\left(\frac{r}{X}\right) = \int_0^\infty f(t) d(t^\alpha).
\]
\end{prop}

\begin{proof}
This can be found in P\'{o}lya and Szeg\H{o}'s book \cite[Problem No. 159 in Part II, Chapter 4 of Volume I]{PolyaSzego1}; see also P\'{o}lya's paper \cite{Polya1923}.
\end{proof}

\begin{proof}[Proof of Proposition \ref{prop:regularlyvarying_variance}]
We define
\[
\F(s) = \sum_{d \in \sgbsf} \frac{1}{d^s} \prod_{\substack{b \nmid d \\  b \in B}} \Big(1 - \frac{2}{b}\Big) = \prod_{b \in B} \Big(1 - \frac{2}{b} + \frac{1}{b^s}\Big),
\]
and
\[
\G(s) = \sum_{r \in \sgb} \frac{1}{r^s} = \prod_{b \in B} \Big(1 - \frac{1}{b^s}\Big)^{-1}.
\]
For both $\F(s)$ and $\G(s)$ the above sum and Euler product converge absolutely for $\Re\, s > \alpha$. Note that for $\Re\, s > \alpha$ we have
\begin{equation}\label{eq:euler_factorization}
\F(s) = \G(s)\U(s),
\end{equation}
where
\[
\U(s) = \prod_{b\in B} \Big(1 - \frac{2}{b} + \frac{2}{b^{1+s}}-\frac{1}{b^{2s}} \Big) = \sum_{c \geq 1} \frac{u_c}{c^s},
\]
where the coefficients $u_c$ are defined by this relation. (The coefficients $u_c$ will be supported on $\sgb$ but this fact will not be important.) 

Note the Euler product defining $\U(s)$ converges absolutely for $\Re \, s > \alpha/2$, and therefore the Dirichlet series also converges absolutely in this region. Hence it follows that for any $\varepsilon > 0$,
\[
\sum_{t < c \leq 2t} |u_c| \ll t^{\frac{\alpha}{2} + \varepsilon},
\]
with the implications
\begin{equation}\label{eq:u_c_bounds}
\sum_{c \leq t}  |u_c| \ll t^{\frac{\alpha}{2}+\varepsilon}, \hspace{20pt} \sum_{c\geq t} \frac{|u_c|}{c} \ll t^{\frac{\alpha}{2} - 1 + \varepsilon},
\end{equation}
which will be important later.

By using the Dirichlet convolution implicit in \eqref{eq:euler_factorization}, we can write
\[
C_2(H) = 2 \sum_{r \in \sgb} \sum_{c} \sum_{\lambda \geq 1} u_c \cdot \Big(\frac{H}{rc}\Big)^2 \Big| V\Big(\frac{H\lambda}{rc}\Big)\Big|^2 = 2 \sum_{r \in \sgb} f\left(\frac{r}{H}\right),
\]
for
\[
f(x) = \sum_{c}\sum_{\lambda \geq 1} \frac{u_c}{(xc)^2} \Big|V\Big(\frac{\lambda}{xc}\Big)\Big|^2.
\]
But from the bound $V(\xi)^2 \ll \min(1, 1/\xi^2)$, we have
\begin{multline*}
f(x) \ll \frac{1}{x^2} \sum_{c} \frac{|u_c|}{c^2} \Big( \sum_{\lambda \leq cx}1 + \sum_{\lambda > cx} \frac{x^2 c^2}{\lambda^2}\Big) 
\ll \sum_{c < 1/x} |u_c| + \frac{1}{x}\sum_{c \geq 1/x} \frac{|u_c|}{c} \ll \begin{cases} \frac{1}{x} & \textrm{if}\; x > 1, \\ \left(\frac{1}{x}\right)^{\frac{\alpha}{2} + \varepsilon} & \textrm{if}\; x \leq 1. \end{cases}
\end{multline*}
The last estimate follows because if $x > 1$ the first term vanishes while the second is bounded trivially, while if $x \leq 1$ both terms satisfy the claimed estimate by \eqref{eq:u_c_bounds}. 

Therefore one may check that P\'{o}lya's proposition may be applied and 
\begin{equation}\label{eq:c2exp}
\begin{split}
C_2(H) &\sim 2 N_{\sgb}(H) \int_0^\infty f(t) \alpha t^{\alpha-1}\, dt\\
&= 2 N_{\sgb}(H) \sum_c \sum_{\lambda \geq 1} \frac{u_c}{c^2} \int_0^\infty \frac{1}{t^2} V\Big(\frac{\lambda}{tc}\Big)^2 \alpha t^{\alpha-1}\, dt,
\end{split}
\end{equation}
where rearrangement of sums and integrals is justified in the second line of \eqref{eq:c2exp} by absolute convergence. By a change of variables $\tau = \lambda/tc$, the last line can be simplified to
\[
2 N_{\sgb}(H)  \zeta(2-\alpha) \prod_{b \in B} \Big(1 - \frac{2}{b} + \frac{2}{b^{1+\alpha}} - \frac{1}{b^{2\alpha}}\Big) \alpha \int_0^\infty \tau^{1-\alpha} V(\tau)^2\, d\tau.
\]
since the sums over $\lambda$ and $c$ can then be simplified as $\zeta(2-\alpha)$ and $\U(\alpha)$ respectively. But (see \cite[formula 3.823]{gradshteyn2014table})
\[
\int_0^\infty \tau^{1-\alpha} V(\tau)^2\, d\tau = - 2^{\alpha-1} \pi^{\alpha-2} \cos\left(\frac{\pi \alpha}{2}\right) \Gamma(-\alpha),
\]
which recovers the constant \eqref{eq:eulerprod_var} claimed in the proposition.
\end{proof}

\begin{rem}
We will not require it, but with a bit more work one can show that if $\sgb$ is regularly varying with index $\alpha \in (0,1)$, and $\vp$ is of bounded variation with compact support,
\[
C_2(H;\vp) \sim A_{\vp,\alpha} N_{\sgb}(H),
\]
where
\[
A_{\vp,\alpha} = 2\alpha \zeta(2-\alpha) \prod_{b \in B} \Big(1 - \frac{2}{b} + \frac{2}{b^{1+\alpha}} - \frac{1}{b^{2\alpha}}\Big) \int_0^\infty \tau^{1-\alpha} |\hat{\vp}(\tau)|^2\, d\tau.
\]
\end{rem}

\section{Diagonal terms}

In this section we show how the approximation of $C_k(H;\vp)$ by Gaussian moments arises from terms in which $r_1,\ldots,r_k$ are \emph{paired} and \emph{non-repeated} in the sum defining this quantity. In the next section we will show that the remaining terms are negligible.

We say that the tuple $r_1,\ldots,r_k$ is \emph{paired} if $k$ is even and we may partition $\{1,2,\ldots,k\}$ into $k/2$ disjoint pairs $\{(a_i, b_i)\}_{i=1}^{k/2}$ with $r_{a_i} = r_{b_i}$ and $a_i \neq b_i$. We say that $r_1,\ldots,r_k$ is \emph{repeated} if $r_{i_1}=r_{i_2} = r_{i_3}$ for some $i_1<i_2<i_3$. For $k=2$ all the terms in $C_2(H;\vp)$ must be paired.

We adopt the abbreviations $r = [r_1,\ldots,r_k]$ for the $\lcm$ and $\mathbf{r} = (r_1,\ldots,r_k)$ for the vector of the $r_i$. Given integers $r_1, r_2, \ldots, r_k>1$ belonging to $\sgbsf$, let
\[ S_H(\mathbf{r}) := \sum_{\substack{\rho_i \in \mc{R}_B(r_i) \, (1 \le i \le k)\\\sum \rho_i \equiv 0 \bmod 1}} \prod_{i=1}^{k} F_H(\rho_i). \]
Our approach in this section largely follows Montgomery and Soundararajan's proof of \cite[Theorem 1]{MS}. We will use the following variant of the Fundamental Lemma; the result is a generalization of \cite[Lemma 2]{MS}, which corresponds to $B$ being the set of primes. The original proof\footnote{In \cite[p.~597]{MS}, the second occurrence of $r_i$ in the first equation should not be there.} works as is.
\begin{lem}[Montgomery and Soundararajan]\label{lem:msgen}
Let $q_1,\ldots,q_k$ be integers with $q_i>1$ and $q_i \in \sgbsf$. Let $G$ be a complex-valued function defined on $(0,1)$, and suppose that $G_0$ is a non-decreasing function on $\sgbsf$ such that
\begin{equation}\label{eq:mshypothesis}
\sum_{a=1}^{q-1} \left|G\left( \frac{a}{q} \right) \right|^2 \le q G_0(q)
\end{equation}
for all $1 \neq q \in \sgbsf$. Then 
\begin{equation}\label{eq:msgen} 
\bigg| \sum_{\substack{a_1,\ldots, a_k\\ 0<a_i<q_i \\ \sum a_i/q_i \equiv 0 \bmod 1}} \prod_{i=1}^{k} G\left( \frac{a_i}{q_i}\right)\bigg| \le \frac{1}{[q_1,\ldots,q_k]} \prod_{i=1}^{k}(q_i G_0(q_i)^{\frac{1}{2}}).
\end{equation}
\end{lem}

 The next lemma separates repeated or non-paired $r$ from what we will show is the main contribution to $C_k(H;\vp)$.
 
\begin{lem}\label{lem:diag}
Let $k \ge 3$, and suppose $\vp$ is of bounded variation and supported in a compact subset of $[0,\infty)$. If $k$ is odd we have
\begin{equation}
 C_k(H;\vp) \ll \sum_{\substack{\mathbf{r} \text{ is repeated}\\\text{or non-paired}\\r_1,\ldots, r_k>1}}
 \prod_{i=1}^{k} \frac{\mobb^2(r_i)}{r_i}S_H(\mathbf{r}).
 \end{equation}
If $k$ is even we have
\begin{equation}
 C_k(H;\vp) = \mu_k \sum_{r_1,\ldots,r_{k/2}>1} \prod_{i=1}^{k/2} g^2(r_i) \sum_{\substack{b_1,\ldots,b_{k/2}\\1 \le b_i \le r_i\\ \sum_{i=1}^{k/2} b_i/r_i \equiv 0 \bmod 1}} \prod_{i=1}^{k} J_H(b_i,r_i) + O \left( \sum_{\substack{\mathbf{r} \text{ is repeated}\\\text{or non-paired}\\r_1,\ldots, r_k>1}}
 \prod_{i=1}^{k} \frac{\mobb^2(r_i)}{r_i}S_H(\mathbf{r}) \right)
 \end{equation}
 where 
\[ J_H(b,n):= \sum_{\substack{a=1\\ (a,n) \, \text{ is $B$-free}\\ (b-a,n) \text{ is $B$-free}}}^{n} \Phi_H\left( \frac{a}{n}\right) \Phi_H\left( \frac{b-a}{n}\right).\]
\end{lem}

\begin{proof}
We first consider odd $k$. There are no vectors $(r_1,\ldots,r_k)$ that are both paired and non-repeated. The proof of this case is concluded by recalling that $g(n)$ is dominated by $\mobb^2(n)/n$ and $\Phi_H$ by $F_H$.

We now consider even $k$. By the triangle inequality,
\begin{equation}\label{eq:ck triangle}
 C_k(H;\vp) = \sum_{\substack{\mathbf{r} \text{ is paired}\\ \text{and non-repeated}\\r_1,\ldots, r_k>1}}
 \prod_{i=1}^{k} g(r_i) \sum_{\substack{\rho_i \in \mc{R}_{B}(r_i) \, (1 \le i \le k)\\ \rho_1 + \cdots + \rho_k \equiv 0 \bmod 1}} \prod_{i=1}^{k} \Phi_H(\rho_i) + O\left(  \sum_{\substack{\mathbf{r} \text{ is repeated}\\\text{or non-paired}\\r_1,\ldots, r_k>1}}
 \prod_{i=1}^{k} \frac{\mobb^2(r_i)}{r_i}S_H(\mathbf{r}) \right).
 \end{equation}
There are $(k-1)(k-3)\cdots 1 =\mu_k$ ways in which the pairing in the first sum in \eqref{eq:ck triangle} can occur. We take the pairing to be $r_i  =r_{k/2 + i}$ without loss of generality. We further write $\rho_i  = a_{i}/r_i$ and set $b_i$ to be the unique integer in $[1,r_i]$ congruent to $a_i  + a_{k/2+i}$ modulo $r_i$. Hence
\begin{align}\label{eq:mt1} 
\sum_{\substack{\mathbf{r} \text{ is paired}\\ \text{and non-repeated}\\r_1,\ldots, r_k>1}}
 \prod_{i=1}^{k} g(r_i) \sum_{\substack{\rho_i \in \mc{R}_{B}(r_i) \, (1 \le i \le k)\\ \rho_1 + \cdots + \rho_k \equiv 0 \bmod 1}} \prod_{i=1}^{k} \Phi_H(\rho_i) &= \mu_k \sum_{\substack{r_1,\ldots,r_{k/2}>1\\\text{distinct}}} \prod_{i=1}^{k/2} g^2(r_i) \sum_{\substack{b_1,\ldots,b_{k/2}\\1 \le b_i \le r_i\\ \sum_{i=1}^{k/2} b_i/r_i \equiv 0 \bmod 1}} \prod_{i=1}^{k} J_H(b_i,r_i)\\
 &= \mu_k \sum_{r_1,\ldots,r_{k/2}>1} \prod_{i=1}^{k/2} g^2(r_i) \sum_{\substack{b_1,\ldots,b_{k/2}\\ 1 \le b_i \le r_i\\ \sum_{i=1}^{k/2} b_i/r_i \equiv 0 \bmod 1}} \prod_{i=1}^{k} J_H(b_i,r_i) \\
 & \qquad + O\left( 
\sum_{\substack{\mathbf{r} \text{ is repeated}\\r_1,\ldots, r_k>1}}
 \prod_{i=1}^{k} \frac{\mobb^2(r_i)}{r_i} S_H(\mathbf{r})\right).
\end{align}
This finishes the proof.
\end{proof}

We now show that paired and non-repeated terms above can be reduced to powers of the variance.

\begin{prop}\label{prop:diag}
Let $k \ge 4$ be even. Suppose $\vp$ satisfies the assumptions of Lemma
\ref{lem:estimate c2rho}.
If the sequence $\sgb$ has index $\alpha$ then
\[ \sum_{r_1,\ldots,r_{k/2}>1} \prod_{i=1}^{k/2} g^2(r_i) \sum_{\substack{b_1,\ldots,b_{k/2}\\1 \le b_i \le r_i\\ \sum_{i=1}^{k/2} b_i/r_i \equiv 0 \bmod 1}} \prod_{i=1}^{k} J_H(b_i,r_i) = C_2(H;\vp)^{\frac{k}{2}} \left( 1 + O\left(H^{-\alpha+o(1)}\right)\right).\]
\end{prop}
\begin{proof}
Let $j$ be the number of values of $i$ for which $b_i \neq r_i$, so that $b_i = r_i$ for the remaining $k/2-j$ values of $i$. Since there are $\binom{k/2}{j}$ ways of choosing the $j$ indices, we see that the left-hand side is 
\begin{equation}
\sum_{j=0}^{k/2} \binom{\frac{k}{2}}{j} C_2(H;\vp)^{\frac{k}{2}-j} W_j(H)
\end{equation}
where $W_0 \equiv 1$ and 
\[ W_j(H) = \sum_{r_1,\ldots,r_j>1} \prod_{i=1}^{j} g^2(r_i) \sum_{\substack{b_1,\ldots,b_j\\ 0<b_i<r_i\\ \sum b_i/r_i \equiv 0 \bmod 1}}\prod_{i=1}^{j} J_H(b_i,r_i).\]
Recall $C_2(H;\vp)=H^{\alpha+o(1)}$ by Lemma~\ref{lem:estimate c2rho}. The term $j=0$ will contribute the main term, so it remains to bound the other terms. It is also clear that $j=1$ contributes $0$ as $b_1/r_1 \in (0,1)$.  
To finish the proof it suffices to show that $W_j(H) = O(H^{\alpha(j-1)+o(1)})$ for $j \ge 2$.

Montgomery and Soundararajan \cite[equation~(34)]{MS} showed that \begin{equation}\label{eq:J bound}
\sum_{b=1}^{n-1} |J_H(b,n)|^2 \ll n^{3}H
\end{equation}
for $B$ being the set of primes and $\vp = \intervalind$. Their argument works as is for general $B$ and general $\vp$. Indeed their proof proceeds by dropping the conditions $(a,n)=1$ and $(b-a,n)=1$ from the sum defining $J$, applying the estimate $\Phi_H \ll F_H$ (in the general case this is Lemma \ref{lem:bdPhibyF}), and straightforwardly estimating the sum of positive terms that result, so their bound applies to our sum as well.

Due to \eqref{eq:J bound} we may apply Lemma~\ref{lem:msgen} with $G_0(n)=CHn^{2}$ to find that
\[ \sum_{\substack{b_1,\ldots,b_j\\ 0< b_i<r_i \\ \sum b_i/r_i \in \Z}} \prod_{i=1}^{j} J_H(b_i,r_i) \ll \frac{1}{[r_1,\ldots,r_j]} \prod_{i=1}^{j} (r_i^{2}H^{\frac{1}{2}}).\]
Additionally, as $|\Phi_H(x)| \ll F_H(x) \le 1/\|x\|$, we have the simple bound
\[ \sum_{b=1}^{n-1}|J_H(b,n)| \ll n^{2} \sum_{a=1}^{n-1} \frac{1}{\min\{a,n-a\} } \sum_{\ss{b=1 \\ b \neq a}}^{n-1}\frac{1}{\min\{|b-a|,n-|b-a|\}} \ll n^{2} (\log n)^2\]
and hence, for any value of $Z \ge 1$,
\begin{align} 
W_j(H) &\ll H^{\frac{j}{2}} \sum_{\substack{r_1,\ldots, r_j\in \sgbsf\\ [r_1,\ldots,r_j]>Z}} \frac{1}{[r_1,\ldots,r_j]} +  \sum_{\substack{r_1,\ldots, r_j \in \sgbsf\\ [r_1,\ldots,r_j]\le Z}} \prod_{i=1}^{j} (\log r_i)^2 \\
& \le H^{\frac{j}{2}} \sum_{r>Z} \frac{\tau_{B,\lcm,j}(r)}{r} + (\log Z)^{2j} \sum_{r \le Z} \tau_{B,\lcm,j}(r)\\
& \le H^{\frac{j}{2}} Z^{\alpha-1+o(1)} + Z^{\alpha+o(1)}.
\end{align}
We take $Z=H^{j/2}$ to find that the left-hand side is \[\ll H^{\frac{j\alpha}{2} + o(1)} \ll H^{\alpha(j-1)+o(1)}\]
as $j \ge 2$.
\end{proof}

Taken together, Lemma \ref{lem:diag} and Proposition \ref{prop:diag} show that paired and non-repeated terms in $C_k(H;\vp)$ are enough to recover the claimed main term.

\section{Off-diagonal terms}

In this section, we show that the repeated or non-paired terms contribute negligibly to $C_k(H;\vp)$. Our main tool will be a refinement of the Fundamental Lemma, due also to Montgomery and Vaughan \cite[Lemma 8]{MV}. However, we will also crucially use ideas of Nunes to bound the range of $\bf{r}$ that we need to consider. We also will critically make use of the P\'{o}lya--Vinogradov inequality to bound a certain term that appears in the refined Fundamental Lemma; this is a new ingredient of our proof and some argument of this sort seems to be essential when the index $\alpha$ is less than or equal to $1/2$ (see Remark \ref{rem:alphaHalf} for a relevant discussion).

\subsection{Preliminary estimates: Nunes's reduction in the range of \bf{r}}\label{sec:Nunes_bound}

Following Nunes \cite{Nunes1}, in this subsection we will show that in the sum defining $C_k(H;\vp)$, those $\bf{r}$ for which $r$ is large yet each $r_i$ is relatively small make a negligible contribution. We first make a few simple observations.

The following bound follows directly from the Fundamental Lemma. (Very similar estimates have been used in \cite[p.~317, equation~(9)]{MV}, \cite[equation~(33)]{Hall2}, and \cite[Lemma~2.3]{Nunes1}.)
\begin{lem}[Montgomery and Vaughan]\label{lem:fl bound}
Given integers $r_1, r_2, \ldots, r_k>1$ belonging to $\sgbsf$ we have
\begin{equation}\label{eq:hall}
S_H(\mathbf{r}) \ll r_1 r_2 \cdots r_k r^{-1} \prod_{i=1}^{k} \min\{r_i,H \}^{\frac{1}{2}} \le H^{\frac{k}{2}} r_1 r_2 \cdots r_k r^{-1} .
\end{equation}
\end{lem}

The following elementary estimate was proved by Nunes \cite[Lemma~2.3]{Nunes1} in the special case $B=\{p^m : p \text{ prime}\}$ ($m \ge 2$).
\begin{lem}[Nunes]\label{lem:nunes bound}
Given $r_1,\ldots, r_k >1$ belonging to $\sgbsf$ we have
\begin{equation}\label{eq:nunes} S_H(\mathbf{r})  \ll \prod_{i=1}^{k} r_i \log r_i. 
\end{equation}
\end{lem}
\begin{proof}
Ignore the restriction $\sum \rho_i \equiv 0 \bmod 1$ and apply the first part of Lemma~\ref{lem:FHmoments}.
\end{proof}
The Fundamental Lemma only sees the $L^2$-norm of $F_H(i/r)$. Lemma~\ref{lem:nunes bound} is superior to Lemma~\ref{lem:fl bound} in certain ranges, as it makes use of the much smaller $L^1$-norm.
As Nunes does, we may combine \eqref{eq:hall} and \eqref{eq:nunes}, obtaining the bound
\begin{equation}\label{eq:combined}
S_H(\mathbf{r}) \ll \left(\prod_{i=1}^{k} r_i
\right) \min \{ H^{\frac{k}{2}}r^{-1}, \prod_{i=1}^{k} \log r_i\} \le \prod_{i=1}^{k} r_i  \cdot \min \{ H^{\frac{k}{2}}r^{-1}, (\log r)^k\}.
\end{equation}
We now use Nunes's bound \eqref{eq:nunes} to deal with $\mathbf{r}$ with small $\lcm$. They turn out to contribute negligibly, a fact that is not detected directly by the Fundamental Lemma bound \eqref{eq:hall}.
\begin{lem}\label{lem:hn1}
Let $H,M>1$. Suppose that $\sgb$ has index $\alpha \in (0,1)$. We have \[\sum_{\substack{r_1,\ldots, r_k>1\\ r \le M}}
\prod_{i=1}^{k} \frac{\mobb^2(r_i)}{r_i}S_H(\mathbf{r}) \ll M^{\alpha+o(1)}.\]
\end{lem}
\begin{proof}
Appealing to \eqref{eq:nunes}, the contribution of $\mathbf{r}$ with $r \le M$ is at most
\[\ll \sum_{\substack{r_1,\ldots, r_k>1\\ r \le M }} \prod_{i=1}^{k} \frac{\mobb^2(r_i)}{ r_i} \prod_{i=1}^{k} r_i \log r_i \le \sum_{\substack{r_1,\ldots, r_k>1\\ r_i \in \sgbsf,\,r \le M }} (\log r)^k  \ll (\log M)^k \sum_{r \le M} \tau_{B,\lcm,k}(r).\]
The claimed bound then follows from Lemma \ref{lem:tau_lem_bound} since $\sgb$ has index $\alpha$.
\end{proof}

We now use the Fundamental Lemma to show that among those $\mathbf{r}$ with $r>M$, the contribution of $\mathbf{r}$ with $r_i\leq N$ is small. Here $M$ and $N$ are parameters to be chosen later.
\begin{lem}\label{lem:hn2}
Let $1<N<H$ and $M>1$. If $\sgb$ has index $\alpha \in (0,1)$ we have
\[\sum_{\substack{r_1,\ldots, r_k>1\\ r_i \le N \text{ for some }i\\r> M}}
 \prod_{i=1}^{k} \frac{\mobb^2(r_i)}{ r_i} S_H(\mathbf{r}) \ll H^{\frac{k-1}{2}}N^{\frac{1}{2}} M^{\alpha-1+o(1)}.\]
\end{lem}
\begin{proof}
If $r_i \le N$ for some $i$, and all $r_j>1$ are in $\sgbsf$ then \eqref{eq:hall} implies that 
\begin{equation}\label{eq:savings}
S_H(\mathbf{r}) \ll H^{\frac{k-1}{2}} r_1 r_2 \cdots r_k r^{-1} N^{\frac{1}{2}},
\end{equation}
so that our sum is at most
\begin{equation}
\sum_{\substack{r_1,\ldots, r_k>1\\ r_i \le N \text{ for some }i\\r>M}} \prod_{i=1}^{k}\frac{\mobb^2(r_i)}{ r_i}   S_H(\mathbf{r}) \le
H^{\frac{k-1}{2}} N^{\frac{1}{2}} \sum_{\substack{r_1,\ldots, r_k>1\\ r_i \le N \text{ for some }i\\r> M}} \frac{\prod_{i=1}^{k}\mobb^2(r_i)}{r}  \le H^{\frac{k-1}{2}}N^{\frac{1}{2}} \sum_{r >M} \frac{\tau_{B,\lcm,k}(r)}{r}.
\end{equation}
The claimed bound again follows from Lemma \ref{lem:tau_lem_bound}, as $\sgb$ has index $\alpha$.
\end{proof}

\subsection{Using the Fundamental Lemma}\label{sec:fl_bound}

To estimate off-diagonal contributions we use the following variant of the Fundamental Lemma. It generalizes \cite[Lemma 7]{MV} of Montgomery and Vaughan, which corresponds to $B$ being the set of primes and $T = H^{1/9}$. The proof follows that of \cite{MV} essentially without change and so we do not include it here.

\begin{lem}[Montgomery and Vaughan]\label{lem:lem7gen}
Let $H>1$ and $1 \le T \le H^{1/9}$. Let $k \ge 3$, and let $r_1$, $r_2$, \ldots, $r_k$ be  integers with $r_i>1$ and $r_i \in \sgbsf$. Further let $r=[r_1,\ldots,r_k]$ and $d=(r_1,r_2)$, and write $d=st$ where $s \mid r_3  r_4 \cdots r_k$ and $(t,r_3 r_4 \cdots r_k)=1$. 
Then
\[ S_H(\mathbf{r}) \ll r_1 r_2 \cdots r_k r^{-1}  H^{\frac{k}{2}} (T_1 + T_2 + T_3 + T_4)\]
where 
\begin{align}
T_1 &= \frac{\log H}{T^{\frac{1}{2}}},\\
T_2&=\begin{cases} d^{-\frac{1}{4}} & \text{if }r_i > H T^{-1} \text{ for all }i,\\ 0 &\text{otherwise,}\end{cases},\\
T_3 &= \begin{cases} s^{-\frac{1}{2}} & \text{if }r_i>H T^{-1}\text{ for all }i\text{ and }r_1=r_2,\\ 0 & \text{ otherwise,}\end{cases},\\
T_4&= \left( \frac{1}{r_1 r_2 s H^2}\sum_{\tau \in \mc{R}_B(t)} F_H\left(  \frac{\|r^{\prime}_1 s \tau \|}{r^{\prime}_1 s} \right)^2 F_H\left(  \frac{\|r^{\prime}_2 s \tau \|}{r^{\prime}_2 s} \right)^2\right)^{\frac{1}{2}} 
\end{align}
if $H T^{-1} \le r_1,r_2 \le H^{2}$, $t>d^{1/2}$ and $d \le (H T)^{1/2}$; otherwise $T_4=0$.
\end{lem}

We use this to produce a first bound on repeated or non-paired $\mathbf{r}$ outside of the range treated by Nunes's bound.

\begin{prop}\label{prop:e2 e3}
Let $k \ge 3$. Let $N$ be in the range $H \ge N \ge H^{8/9}$ and $M>1$. Suppose that $\sgb$ has index $\alpha  \in (0,1)$. We have
\begin{equation}
    \sum_{\substack{\mathbf{r} \text{ is repeated}\\\text{or non-paired}\\r_1,\ldots, r_k>N\\ r>M}}
 \prod_{i=1}^{k} \frac{\mobb^2(r_i)}{r_i} S_H(\mathbf{r}) \ll H^{\frac{k-1}{2}+o(1)} \left(  N^{\frac{1}{2}} + H^{\frac{1}{2}}  N^{-\frac{1}{4(k-1)}} + \mc{E}^{\frac{1}{2}}H^{-\frac{1}{2}} \right) M^{\alpha-1+o(1)} 
\end{equation}
where
\[ \mc{E} =   \sum_{\substack{N^{1/(k-1)} \le st \le H/\sqrt{N}\\  (s,t)=1,\,s,t \in \sgbsf,\,t>s}} s^{-\alpha-3}t^{-\alpha-2} \sum_{\tau \in \mc{R}_{B}(t)}\bigg( \sum_{\substack{a: \, (a,t)=1\\  N<ast \le H^2 \\ a \in \sgbsf}} a^{-\alpha-1} F_H\left(  \frac{\|a s \tau \|}{a s} \right)^2 \bigg)^2.\]
\end{prop}
\begin{proof}
Let $\mathbf{r}$ be a vector which is either repeated or non-paired with $r > M$ and such that $r_i > N$ for all $i$. It suffices to bound
\[ \sum_{\substack{\mathbf{r} \text{ is repeated}\\r_1,\ldots,r_k > N\\r > M}} \prod_{i=1}^{k} \frac{\mobb^2(r_i)}{r_i} S_H(\mathbf{r})+\sum_{\substack{\mathbf{r} \text{ is non-paired}\\\text{and non-repeated}\\r_1,\ldots,r_k > N\\r > M}} \prod_{i=1}^{k} \frac{\mobb^2(r_i)}{r_i} S_H(\mathbf{r}).\]
If $\mathbf{r}$ is repeated with $r_{i_1}=r_{i_2} = r_{i_3}$, then we apply
Lemma~\ref{lem:lem7gen} with $r_{i_1},r_{i_2}$ in place of $r_1,r_2$, obtaining
\[ S_H(\mathbf{r}) \ll r_1 r_2 \cdots r_k r^{-1}  H^{\frac{k}{2}} (T_1 + T_2 + T_3 + T_4). \] 
(The parameter $T$ is taken to be $H/N$.) To study the $T_i$s, recall that $d=(r_{i_1},r_{i_2}) = r_{i_1}$ factors as $d=st$, where \[s \mid \prod_{i \neq i_1, i_2} r_i\quad\text{ and }\quad\big(t,\prod_{i \neq i_1, i_2} r_i\big)=1.\]
In our case, $t=1$ and $s=d=r_{i_1} >N$. We have \[T_1 = \frac{\log H}{(\frac{H}{N})^{\frac{1}{2}}},\quad T_2 = d^{-\frac{1}{4}} < N^{-\frac{1}{4}},\quad T_3=s^{-\frac{1}{2}} = d^{-\frac{1}{2}} < N^{-\frac{1}{2}}\quad\text{ and }\quad T_4 = 0,\]
so the total contribution of the $T_i$'s in the repeated case is at most 
\begin{equation}\label{eq:t1t2t3} \ll H^{\frac{k}{2}}\left( N^{-\frac{1}{4}}+ \frac{N^{\frac{1}{2}} \log H }{H^{\frac{1}{2}}}\right)  \sum_{\substack{r_1,\ldots,r_k>N\\r_i \in \sgbsf,\, r>M}} \frac{1}{r}  \ll H^{\frac{k-1}{2} + o(1)}   N^{\frac{1}{2}} \sum_{r > M} \frac{\tau_{B,\lcm,k}(r)}{r},
\end{equation}
which is absorbed in the error term since the series over $r$ is $\ll M^{\alpha-1+o(1)}$. 

Suppose that $\mathbf{r}$ is non-paired and also non-repeated. The contribution of $\mathbf{r}$ is $0$ if there is a prime $p$ dividing only one of the $r_i$ (as then $S_H(\mathbf{r})=0$), so we assume that each prime divisor of $r$ divides at least two of the $r_i$. This implies that $r_i \mid \prod_{j \neq i}r_j$. As in \cite[p.~323]{MV}, this implies $r_i \le \prod_{j \neq i} (r_j,r_i)$ and so for each $i$ there exists $j\neq i$ such that \[(r_j,r_i) \ge r_i^{\frac{1}{k-1}} > N^{\frac{1}{k-1}}.\]
We claim that there is at least one pair $r_i,r_j$ with $i \neq j$, $(r_j,r_i) > N^{1/(k-1)}$ and $r_i \neq r_j$. Indeed, if there is no such pair, it means that each value in the multiset $\{ r_i : 1 \le i \le k\}$ appears at least twice. We also know that each value appears at most twice, as we are in the non-repeated case. Hence each value appears twice, contradicting the fact that we are in the non-paired case. 

Hence necessarily $r_{i_1} \neq r_{i_2}$ with $(r_{i_1},r_{i_2}) > N^{1/(k-1)}$ for some $i_1 \neq i_2$. We again apply
Lemma~\ref{lem:lem7gen} with $r_{i_1},r_{i_2}$ in place of $r_1,r_2$, and $T=H/N$. By definition, $T_3=0$. As we have $T_1=(\log H)/(H/N)^{1/2}$ and  $T_2 = d^{-1/4}<N^{-1/(4(k-1))}$, we get that the total contribution of $T_1$, $T_2$ and $T_3$ in the non-paired case is at most
\begin{equation}\label{eq:t1t2t32} \ll H^{\frac{k}{2}}\left( N^{-\frac{1}{4(k-1)}} + \frac{(\log H) N^{\frac{1}{2}}}{H^{\frac{1}{2}}}\right)\sum_{r > M} \frac{\tau_{B,\lcm,k}(r)}{r},
\end{equation}
which is also absorbed in the error term.

We now treat the contribution of $T_4$ when $\mathbf{r}$ is non-paired and non-repeated, and show that it is at most $H^{k/2-1+o(1)}M^{\alpha-1+o(1)}\mc{E}^{1/2}$, so is absorbed as well. 

Assuming always that Lemma~\ref{lem:lem7gen} is applied with the first two elements of $\mathbf{r}$ (at the cost of a constant of size $k^2$ from permuting the indices), and recalling that if $\mathbf{r}$ is non-paired and non-repeated we may assume $(r_1,r_2)>N^{1/(k-1)}$,
we see that $T_4$ contributes at most
\begin{equation}\label{eq:t4cs}
\ll H^{\frac{k}{2}} \sum_{\mathbf{r} \in X}\frac{T_4(\mathbf{r})}{r} \le H^{\frac{k}{2}} \left(\sum_{\mathbf{r} \in X} \frac{T_4^2}{r}\right)^{\frac{1}{2}}\left(\sum_{\mathbf{r} \in X} \frac{1}{r}\right)^{\frac{1}{2}}
\end{equation}
by Cauchy--Schwarz, where
\[X = \{ \mathbf{r} \in \N_{>1}^k : N< r_1,r_2 \le H^{2}, \, N^{\frac{1}{k-1}} < d \le \frac{H}{\sqrt{N}}, t>d^{\frac{1}{2}},\,  r_i > N, \, r >M, \, r_i \in \sgbsf\}\]
and $d=(r_1,r_2)$, $s=(d,r_3 \cdots r_k)$, $t=d/s$. (Note that the condition $t>d^{1/2}$ is the same as $t>s$.) 
The second sum is at most  
\begin{equation}\label{eq:t4 second}
\sum_{r>M} \frac{\tau_{B,\lcm,k}(r)}{r} \ll M^{\alpha-1+o(1)}.
\end{equation}
To study the first sum, we write $r$ as $r^{\prime}_1 r^{\prime}_2 s t w$ where $r^{\prime}_i = r_i/d$ and $w = r/(r^{\prime}_1 r^{\prime}_2 s t)$. Note that $r^{\prime}_1$, $r^{\prime}_2$, $s$ and $t$ are pairwise coprime so $w$ must be an integer. Instead of summing over $r_1$ and $r_2$ we sum over $r^{\prime}_1$, $r^{\prime}_2$, $s$ and $t$ (so $r_1$ and $r_2$ are determined). Given $r_1$, $r_2$ and $w$ we have $r$; given $r$, $r_1$ and $r_2$ there are at most $\tau^{k-2}(r)$ possibilities for $\mathbf{r}$ with these values of $r_1$, $r_2$ and $r$, since each $r_i$ ($3 \le i \le k$) divides $r=wr^{\prime}_1 r^{\prime}_2 st$. Here $\tau$ is the usual divisor function. Hence we have
\begin{align}
\sum_{\mathbf{r} \in X} \frac{T_4^2}{r} &\le \sum_{\substack{r^{\prime}_1, r^{\prime}_2, s,t\in\sgbsf \text{ coprime}, \, t>s\\ N < r^{\prime}_i st \le H^{2}\, (i=1,2)\\ N^{1/(k-1)} \le st \le  H/\sqrt{N}}} \frac{1}{r^{\prime}_1 r^{\prime}_2 st} T_4^2 \sum_{\substack{w> M/(r^{\prime}_1 r^{\prime}_2 st)\\w\in\sgbsf\\(w,s)=1}} \frac{\tau^{k-2}(wr^{\prime}_1 r^{\prime}_2 st)}{w}\\
& \le \sum_{\substack{r^{\prime}_1, r^{\prime}_2, s,t\in\sgbsf \text{ coprime}, \, t>s\\ N < r^{\prime}_i st \le H^{2}\, (i=1,2)\\ N^{1/(k-1)} \le st \le  H/\sqrt{N}}} \frac{\tau^{k-2}(r^{\prime}_1 r^{\prime}_2 st)}{r^{\prime}_1 r^{\prime}_2 st} T_4^2 \sum_{\substack{w> M/(r^{\prime}_1 r^{\prime}_2 st)\\w\in\sgbsf}} \frac{\tau^{k-2}(w)}{w}.
\end{align}
Because $\sgbsf$ has index $\alpha$ and $\tau(n)=n^{o(1)}$, the innermost sum above is
\[ \ll \left(\frac{M}{r^{\prime}_1 r^{\prime}_2 st}\right)^{\alpha-1+o(1)},\]
so that
\begin{equation}
\sum_{\mathbf{r} \in X} \frac{T_4^2}{r} \le M^{\alpha-1+o(1)} H^{o(1)} \sum_{\substack{r^{\prime}_1, r^{\prime}_2, s,t\in\sgbsf \text{ coprime}, \, t>s\\ N < r^{\prime}_i st \le H^{2}\, (i=1,2)\\ N^{1/(k-1)} \le st \le  H/\sqrt{N}}} \frac{T_4^2}{(r^{\prime}_1 r^{\prime}_2 st)^{\alpha}}  .
\end{equation}
Plugging the definition of $T_4$ in the last equation, and first summing over $s,t$ and only later over $r_i^\prime$ we obtain
\begin{equation}
\sum_{\mathbf{r} \in X} \frac{T_4^2}{r} \ll M^{\alpha-1+o(1)} H^{o(1)-2}  \sum_{\substack{N^{1/(k-1)} \le st \le H/\sqrt{N}\\  (s,t)=1,\,s,t \in \sgbsf,\, t>s}} s^{-\alpha-3}t^{-\alpha-2} \sum_{\tau \in \mc{R}_{B}(t)}\bigg( \sum_{\substack{a: \, (a,t)=1\\  N<ast \le H^2 \\ a \in \sgbsf}} a^{-\alpha-1} F_H\left(  \frac{\|a s \tau \|}{a s} \right)^2 \bigg)^2.
\end{equation}
Plugging this bound in \eqref{eq:t4cs}, we end up with $H^{k/2-1+o(1)}M^{\alpha-1+o(1)}\mc{E}^{1/2}$.
\end{proof}

In order to get a good upper bound on $\mc{E}$ we must estimate the frequency with which $\| \|as\tau\|/(as)\|$ is smaller than $1/H$. We will do this by reduction to congruence conditions and we will bound sums over such congruence conditions using the following simple consequence of the P\'{o}lya--Vinogradov inequality. (For P\'{o}lya--Vinogradov see e.g. \cite[Theorem 12.5]{IK}).

\begin{lem}\label{lem:pv application}
Let $\chi$ be a Dirichlet character modulo $q$ and suppose $A \gg q$. We have
\[ \left|\sum_{1 \le |i| \le q/2} \chi(i)F_H\left(\frac{i}{A}\right)^2\right| \ll \begin{cases} AH & \text{if }\chi \text{ is principal,}\\   \min\left\{AH, H^2\sqrt{q} \log q\right\} & \text{if }\chi \text{ is non-principal.}\end{cases}\]
\end{lem}
\begin{proof}
By the definition of $F_H$,
\begin{align}
\left|\sum_{1 \le |i| \le q/2} \chi(i)F_H\left(\frac{i}{A}\right)^2\right| \ll H^2 \left|\sum_{1 \le  |i| \le A/H} \chi(i)\right| + A^2 \left|\sum_{|i|>A/H} \frac{\chi(i)}{i^2}\right|.
\end{align}
The bound for principal characters is evident, so we now consider non-principal ones. The sum over smaller $|i|$ satisfies the required bound by combining the P\'{o}lya--Vinogradov inequality \[\sum_{1 \le |i| \le n} \chi(i) = O(\sqrt{q} \log q)\] with  the trivial bound $|\chi| \le 1$. To bound the sum over larger $i$, we can either use the trivial bound $\ll H/A$, or else again appeal to P\'{o}lya--Vinogradov together with partial summation as follows:
\[
\left|\sum_{i > N} \frac{\chi(b)}{b^2}\right| \ll \frac{1}{N^2} \left|\sum_{i \le N} \chi(i)\right| + \int_{N}^{\infty} \frac{\left|\sum_{n \le y} \chi(n)\right|}{y^3}dy \ll \frac{\sqrt{q} \log q}{N^2} ,
\]
where $N=A/H$.
\end{proof}

In what follows we use the notation $a \sim x$ to mean $x < a \leq 2x$. The following proposition will be used shortly.

\begin{prop} \label{prop:GrvBd}
Suppose $\sgb$ is of index $\alpha \in (0,1)$. Let  $s,t \in \sgbsf$ be coprime positive integers with $t \le H$. Let $r$ be a $B$-free divisor of $t$. Given $A \le H^2$ let
\[ G_{r,v}(\tau_1):= A^{-(\alpha+1)} \sum_{\substack{1 \le |i| \le t/(2r)\\ (i,t/r)=1}}  F_H\left(\frac{i}{2Ast/r} \right)^2 \sum_{\substack{a \sim A, \, a \in \sgbsf,\,(a,t)=1\\a \equiv  i s^{-1}(\tau_1/r)^{-1} \bmod t/r}} 1\]
for $\tau_1$ with $(t,\tau_1)=r$. We have
\begin{equation}\label{eq:pv corollary}
\sum_{\substack{\tau_1/t \in \mc{R}_{B}(t)\\ (t,\tau_1)=r}} G_{r,v}(\tau_1)^2\ll \frac{H^{2}}{\phi\left(\frac{t}{r}\right)} \left( \left(\frac{st}{r}\right)^2 A^{o(1)}+  A^{-2(\alpha+1)} 
\min \left\{  H^2 \frac{t^{1+o(1)}}{r}, \left(\frac{As t}{r}\right)^2\right\}
\sum_{\chi \bmod t/r} \left|\sum_{\substack{a\sim A, \, a \in \sgbsf\\ (a,t)=1}}\chi(a)\right|^2 \right).
\end{equation}
\end{prop}
\begin{proof}
Expanding the square we have
\begin{align}
 \sum_{\substack{\tau_1/t \in \mc{R}_{B}(t)\\ (t,\tau_1)=r}} G_{r,v}(\tau_1)^2&\ll  A^{-2(\alpha+1)} \sum_{\substack{1 \le |i_1|,|i_2| \le t/(2r)\\ (i_1 i_2,t/r)=1}} F_H\left(\frac{i_1}{2Ast/r}\right)^2 F_H\left(\frac{i_2}{2Ast/r}\right)^2  \sum_{\substack{\tau_1/t \in \mc{R}_{B}(t)\\ (t,\tau_1)=r}} \sum_{\substack{a_1,a_2 \sim A,\, a_1,a_2 \in \sgbsf \\ (a_1 a_2,t) = 1 \\ a_1 \equiv i_1 s^{-1} (\tau_1/r)^{-1} \bmod t/r\\ a_2 \equiv i_2 s^{-1} (\tau_1/r)^{-1} \bmod t/r }} 1.
\end{align}
The two congruences in the innermost sum imply \begin{equation}\label{eq:new cong}
a_1 i_2 \equiv a_2 i_1 \bmod \frac{t}{r},
\end{equation}
and we may replace the inner double sum over $\tau_1$, $a_1$ and $a_2$ with the sum
\[   \sum_{\substack{a_1,a_2 \sim A,\, a_1,a_2 \in \sgbsf \\ (a_1 a_2,t) = 1 \\ a_1 i_2 \equiv a_2 i_1 \bmod t/r}}1.\]
We can detect \eqref{eq:new cong}  using orthogonality of characters, obtaining 
\begin{align}
 \sum_{\substack{\tau_1/t \in \mc{R}_{B}(t)\\ (t,\tau_1)=r}} G_{r,v}(\tau_1)^2 \ll& \frac{A^{-2(\alpha+1)}}{\phi\left(\frac{t}{r}\right)} \sum_{\chi \bmod t/r} \left( \sum_{1 \le |i_1| \le t/(2r)} \chi(i_1) F_H\left(\frac{i_1r}{2Ast}\right)^2 \sum_{\substack{a_1 \sim A,\, a_1 \in \sgbsf \\ (a_1,t) = 1 }} \bar{\chi}(a_1)\right) \\
&\times \left( \sum_{1 \le |i_2| \le t/r}\bar{\chi}(i_2) F\left(\frac{i_2r}{2Ast}\right)^2  \sum_{\substack{a_2 \sim A, \, a_2 \in \sgbsf\\(a_2,t)=1}} \chi(a_2)\right) \\
=& \frac{1}{\phi\left(\frac{t}{r}\right)} \sum_{\chi \bmod t/r} A^{-2(\alpha+1)} \left| \sum_{1 \le |i| \le t/(2r)}\chi(i) F_H\left(\frac{ir}{2Ast}\right)^2 \right|^2 \left|\sum_{\substack{a \sim A, \, a \in \sgbsf\\(a,t)=1}} \chi(a)\right|^2.
\end{align}
By Lemma~\ref{lem:pv application}, the contribution of the principal character $\chi= \chi_0$ is
\[ \ll \frac{A^{o(1)}}{\phi\left(\frac{t}{r}\right)} \left(\frac{H st}{r}\right)^2, \]
which gives the first term in the required bound. We now consider the non-principal characters. Applying the pointwise bound for the sum of $F$ twisted by $\chi$ as given in Lemma~\ref{lem:pv application}, we see that they contribute 
\[
\ll  H^{2} A^{-2(\alpha+1)} \min \left\{ \frac{H^2 t}{r} \log^2 t, \left(\frac{Ast}{r}\right)^2\right\}\frac{1}{\phi\left(\frac{t}{r}\right)}\sum_{\chi_0 \neq \chi \bmod t/r} \left| \sum_{\substack{a \sim A, \, a \in \sgbsf\\(a,t)=1}} \chi(a)\right|^2.
\]
This gives the second contribution to the bound, and we are done.
\end{proof}
\begin{prop}\label{prop:t4}
Let $N$ be in the range $H \ge N \ge H^{8/9}$. Suppose that 
$\sgb$ has index $\alpha \in (0,1)$. In the notation of Proposition~\ref{prop:e2 e3}, we have
\[ \mc{E} \ll H^{2+o(1)} \left(\left(\frac{H}{N^{\frac{3}{2}}}\right)^{\alpha} + N^{-\frac{\alpha}{k-1}}\right).\]
\end{prop}
\begin{proof}
We dyadically decompose the inner sum over $a$ in the definition of $\mc{E}$.
Given coprime integers $s, t \in \sgbsf$ with $t>1$ and $\tau=\tau_1/t \in \mc{R}_{B}(t)$, and setting $r:=(t,\tau_1)$ for convenience, we have
\begin{align}\label{eq:sum a in u} \sum_{\substack{a \sim U, \, a \in \sgbsf\\ (a,t)=1}} F_H \left( \frac{\| as \tau \|}{a s} \right)^2 &\ll \sum_{\substack{1 \le  i \le t/(2r)\\(i,t/r)=1}}  F_H\left(\frac{i}{2 Ust/r} \right)^2 \sum_{\substack{a \sim U,\, a\in \sgbsf,\,(a,t)=1\\ a \equiv  \pm i s^{-1}(\tau_1/r)^{-1} \bmod t/r}} 1\\
& \ll \sum_{\substack{1 \le  |i| \le t/(2r)\\(i,t/r)=1}}  F_H\left(\frac{i}{2 Ust/r} \right)^2 \sum_{\substack{a \sim U,\, a \in  \sgbsf,\, (a,t)=1\\ a \equiv  i s^{-1}(\tau_1/r)^{-1} \bmod t/r}} 1
\end{align}
for any positive integer $U$, as we now explain. 

First, we may replace \[F_H\left( \frac{\|as \tau\|}{as}\right) \quad \text{ with }\quad F_H\left(\frac{\| as \tau \|}{2Us}\right)\]
as in general if $0<x,y \le 1/2$ satisfy $y \ll x$ then $F_H(x) \ll F_H(y)$. The denominator of the fraction $\| a s \tau \|$ in reduced form is exactly $t/(t,\tau_1)$, since both $s$ and $a$ are coprime with $t$. 
We write the left-hand side of \eqref{eq:sum a in u} as a sum over the possible values of $\| as \tau \|$, and need to count the number of times a given value is obtained, that is, count solutions to $\| a s \tau\| =  i /(t/r)$, which is an equation that determines $a$ modulo $t/r$ up to a sign, yielding the new inner sum over $a$.

Putting this in the definition of $\mc{E}$ and applying the Cauchy--Schwarz inequality, we thus obtain,
\begin{align} 
\mc{E} 
 & \ll H^{o(1)} \sum_{\substack{N^{1/(k-1)} \le st \le H/\sqrt{N}\\  (s,t)=1,\,s,t \in \sgbsf,\,t>s}} s^{-\alpha-3}t^{-\alpha-2} 
 \sum_{r \mid t} \sum_{2^v \in (N/(2st),H^2/(st))} \sum_{\substack{\tau_1/t \in \mc{R}_{B}(t)\\ (t,\tau_1)=r}}  G^2_{r,v}.
\end{align}
Now, by orthogonality we have
\[
\frac{1}{\phi\left(\frac{t}{r}\right)}\sum_{\chi \bmod{t/r}} \bigg|\sum_{\substack{a \sim 2^v\\ a \in \sgbsf}} \chi(a)\bigg|^2 = \sum_{\ss{a_1,a_2 \sim 2^v \\ a_1,a_2\in \sgbsf \\ (a_1a_2,t/r) = 1}} \mathbf{1}_{a_1 \equiv a_2 \bmod{t/r}}.
\]
Thus, using this in the conclusion of Proposition \ref{prop:GrvBd} with $A=2^v$, we get
\[
\sum_{\substack{\tau_1/t \in \mc{R}_{B}(t)\\ (t,\tau_1)=r}}  G^2_{r,v} \ll H^{2}\left(\frac{s^2 \left(\frac{t}{r}\right)^2}{\phi\left(\frac{t}{r}\right)} H^{o(1)} + 2^{-2v} \min\left\{\frac{H^2t}{r}, \left(\frac{2^v st}{r}\right)^2 \right\} \cdot \frac{1}{2^{2\alpha v}} \sum_{\ss{a_1,a_2 \sim 2^v \\ a_1,a_2 \in \sgbsf}} \mathbf{1}_{a_1 \equiv a_2 \bmod{t/r}}\right).
\]
Inserting this estimate back into our upper bound for $\mc{E}$, we obtain a bound \[\mc{E} \ll H^{2+o(1)}(\mc{E}_1 + \mc{E}_2 + \mc{E}_3),\]
where
\begin{align*}
\mc{E}_1 :=& \sum_{\ss{N^{1/(k-1)} < st \leq H/\sqrt{N}\\ (s,t) = 1,\, t > s,\, s,t \in \sgbsf}} (ts)^{-\alpha-1} \sum_{r\mid t} r^{-1} , \\
\mc{E}_2 :=& \sum_{\ss{N^{1/(k-1)} < st \leq H/\sqrt{N} \\ (s,t) = 1,\, t > s,\, s,t \in \sgbsf}} s^{-\alpha-3}t^{-\alpha-2}\sum_{r\mid t} \sum_{N/(2st) < 2^v \leq H^2/(st)} 2^{-2(\alpha+1)v} \min\left\{\frac{H^2t}{r}, \left(\frac{2^v st}{r}\right)^2\right\} \sum_{\ss{a_1, a_2 \sim 2^v \\ a_1,a_2 \in \sgbsf \\ a_1 = a_2}} 1, \\
\mc{E}_3 :=& \sum_{r \leq H/\sqrt{N}} \sum_{\ss{s \in \sgbsf \\ s \leq H/(r\sqrt{N})}} s^{-\alpha-3} \sum_{N^{3/2}/(2H) < 2^v \leq H^2/N^{1/(k-1)}} 2^{-2(\alpha+1)v} \\
&\times \sum_{\ss{ a_1,a_2 \sim 2^v \\ a_1,a_2 \in \sgbsf \\ a_1 \neq a_2}}\quad\sum_{\ss{\ell : \, r\ell \in \sgbsf \\ \max\left\{N^{1/(k-1)},N/2^{v+1}\right\} < rs\ell \leq \min\left\{H/\sqrt{N}, H^2/2^v\right\}}} (r\ell)^{-\alpha-2} \min\left\{H^2\ell, (2^v s\ell)^2\right\} \mathbf{1}_{\ell\mid (a_1-a_2)},
\end{align*}
where the variable $\ell$ in $\mc{E}_3$ represents the value of $t/r$.
We estimate each of these terms in sequence. Utilizing the index $\alpha$ of $\sgbsf$, we easily bound
\[
\mc{E}_1 \ll H^{o(1)}\sum_{\ss{N^{1/(k-1)} < m \leq H/\sqrt{N} \\ m \in \sgbsf}} \tau(m)m^{-\alpha-1} \ll H^{o(1)} \cdot (N^{\frac{1}{k-1}})^{\alpha - (\alpha+1)+o(1)} \ll H^{o(1)} N^{-\frac{1}{k-1}}
\]
where $\tau$ is the usual divisor function.
To treat $\mc{E}_2$ we split the range of $2^v$ at $H/(s(t/r)^{1/2})$ for $s,t$ and $r\mid t$ given, which yields
\begin{align*}
\mc{E}_2 &\ll H^{o(1)}\sum_{\ss{N^{1/(k-1)} < st \leq H/\sqrt{N} \\ (s,t) = 1,\, t > s,\,s,t \in \sgbsf}} s^{-\alpha-3}t^{-\alpha-2} \sum_{r\mid t}\left( \sum_{\frac{N}{2st} < 2^v \leq \frac{H}{s(t/r)^{1/2}}} 2^{-(2+\alpha)v} \left(\frac{2^v st}{r}\right)^2 + \sum_{\frac{H}{s(t/r)^{1/2}} < 2^v \leq \frac{H^2}{st}} 2^{-(2+\alpha)v} \frac{H^2t}{r}\right) \\
&\ll H^{o(1)}\sum_{\ss{N^{1/(k-1)} < st \leq H/\sqrt{N} \\ (s,t) = 1,\, t > s,\, s,t \in \sgbsf}} s^{-\alpha-1}t^{-\alpha} \sum_{r\mid t} \frac{1}{r^2} \left(\frac{st}{N}\right)^{\alpha} + H^{2+o(1)}\sum_{\ss{N^{1/(k-1)} < st \leq H/\sqrt{N} \\ (s,t) = 1,\, t > s,\,s,t \in \sgbsf}} s^{-\alpha-3}t^{-\alpha-1} \sum_{r\mid t}\frac{1}{r} \sum_{2^v > \frac{H}{s(t/r)^{1/2}}} 2^{-(2+\alpha)v} \\
&\ll H^{o(1)}N^{-\alpha} \sum_{\ss{s \leq H/\sqrt{N} \\ s \in \sgbsf}} s^{-1} \sum_{\ss{t \leq H/(s\sqrt{N}) \\ t \in \sgbsf}} 1 + H^{-\alpha+o(1)} \sum_{\ss{N^{1/(k-1)} < st \leq H/\sqrt{N} \\ (s,t) = 1,\, t > s,\,s,t \in \sgbsf}} s^{-1} t^{-\frac{\alpha}{2}} \sum_{r\mid t} r^{-2-\frac{\alpha}{2}} \\
&\ll H^{o(1)} \left(\left(\frac{H}{N^{\frac{3}{2}}}\right)^{\alpha} +  H^{-\alpha} \left(\frac{H}{\sqrt{N}}\right)^{\frac{\alpha}{2}}\right) \ll H^{o(1)} \left(\frac{H}{N^{\frac{3}{2}}}\right)^{\alpha}.
\end{align*}
It remains to treat $\mc{E}_3$. In this case, we split the range according to the condition $\ell \leq (H/(s2^v))^{2}$ as, when this inequality holds, \[\min\{H^2\ell, (2^v s \ell)^2\} = 2^{2v} s^2 \ell^2.\] Hence, we further bound $\mc{E}_3 \ll \mc{E}_3' + \mc{E}_3''$, where we define
\begin{align*}
\mc{E}_3' :=&  \sum_{r \leq H/\sqrt{N}}r^{-\alpha-2} \sum_{\ss{s \in \sgbsf \\ s \leq H/(r\sqrt{N})}} s^{-\alpha-1} \sum_{\frac{N^{3/2}}{2H} < 2^v \leq \frac{H^2}{N^{1/(k-1)}}} \frac{1}{2^{2\alpha v}} \\
&\times \sum_{\ss{ a_1,a_2 \sim 2^v \\ a_1,a_2 \in \sgbsf \\ a_1 \neq a_2}}\quad \sum_{\ss{\ell : \, r\ell \in \sgbsf \\ \max\left\{\frac{N^{1/(k-1)}}{rs},\frac{N}{rs2^{v+1}}\right\} < \ell \leq \min\left\{\frac{H}{rs\sqrt{N}}, \left(\frac{H}{s2^v}\right)^{2}\right\} \\ \ell > s/r}} \ell^{-\alpha} \mathbf{1}_{\ell\mid (a_1-a_2)}, \\
\mc{E}_3'' :=& H^2 \sum_{r \leq H/\sqrt{N}} r^{-\alpha-2} \sum_{\ss{s \in \sgbsf \\ s \leq H/(r\sqrt{N})}} s^{-\alpha-3} \sum_{\frac{N^{3/2}}{2H} < 2^v \leq \frac{H^2}{N^{1/(k-1)}}} 2^{-2v} \cdot \frac{1}{2^{2\alpha v}} \\
&\times \sum_{\ss{ a_1,a_2 \sim 2^v \\ a_1,a_2 \in \sgbsf \\ a_1 \neq a_2}}\quad \sum_{\ss{\ell : \, r\ell \in \sgbsf \\ \max\left\{\frac{N^{1/(k-1)}}{rs},\left(\frac{H}{s2^{v}}\right)^{2}\right\} < \ell \leq \min\left\{\frac{H}{rs\sqrt{N}}, \frac{H^2}{rs2^v}\right\} \\ \ell > s/r}} \ell^{-\alpha-1} \mathbf{1}_{\ell\mid (a_1-a_2)}.
\end{align*}
Consider first $\mc{E}_3'$. Note that the number of integers $\ell$ dividing  $a_1-a_2$ is \[\ll |a_1-a_2|^{o(1)} = H^{o(1)}.\] The inner sum in $\mc{E}_3'$ is therefore \[\ll H^{o(1)}\min\left\{\frac{rs}{N^{\frac{1}{k-1}}}, \frac{2^vrs}{N}\right\}^{\alpha}.\] 

Bounding the minimum trivially by $rs/N^{1/(k-1)}$, we obtain
\begin{align*}
 \mc{E}_3' &\ll H^{o(1)} \sum_{r \leq H/\sqrt{N}}r^{-\alpha-2} \sum_{\ss{s \in \sgbsf \\ s \leq H/(r\sqrt{N})}} s^{-\alpha-1} \sum_{N^{3/2}/(2H) < 2^v \leq H^2/N^{1/(k-1)}}  \frac{(rs)^{\alpha}}{N^{\frac{\alpha}{k-1}}} \\
&\ll H^{o(1)} N^{-\frac{\alpha}{k-1}}.
\end{align*}
Finally, we treat $\mc{E}_3''$ in a similar way, using a divisor bound to count $\ell\mid (a_1-a_2)$. This leads to
\begin{align*}
\mc{E}_3'' \ll H^{2+o(1)} \sum_{r \leq H/\sqrt{N}}r^{-\alpha-2} \sum_{\ss{s \in \sgbsf \\ s \leq H/\sqrt{N}}} s^{-\alpha-3} \sum_{N^{3/2}/(2H) < 2^v \leq H^2/N^{1/(k-1)}} 2^{-2v} \min\left\{\frac{rs}{N^{\frac{1}{k-1}}}, \left(\frac{s2^v}{H}\right)^{2}\right\}^{\alpha + 1}.
\end{align*}
Using the bound $\min\{A,B\}^{\alpha+1} \leq A^\alpha B$, we obtain
\begin{align*}
\mc{E}_3'' &\ll H^{2+o(1)} \sum_{r \leq H/\sqrt{N}}r^{-\alpha-2} \sum_{\ss{s \in \sgbsf \\ s \leq H/\sqrt{N}}} s^{-\alpha-3} \sum_{N^{3/2}/(2H) < 2^v \leq H^2/N^{1/(k-1)}} 2^{-2v} \frac{(rs)^\alpha}{N^{\frac{\alpha}{k}-1}} \left(\frac{s 2^v}{H}\right)^2 \\
&\ll \frac{H^{o(1)}}{N^{\frac{\alpha}{k-1}}} \sum_{r \leq H/\sqrt{N}} \frac{1}{r^2} \sum_{\ss{s \in \sgbsf \\ s \leq H/\sqrt{N}}} \frac{1}{s} \ll H^{o(1)} N^{-\frac{\alpha}{k-1}}.
\end{align*}
It follows then that $\mc{E}_3 \ll H^{o(1)} N^{-\alpha/(k-1)}$. Combining this with our prior estimates for $\mc{E}_1$ and $\mc{E}_2$, we obtain
\[
\mc{E} \ll H^{2+o(1)} \left(N^{-\frac{1}{k-1}} + \left(\frac{H}{N^{\frac{3}{2}}}\right)^{\alpha} + N^{-\frac{\alpha}{k-1}}\right) \ll H^{2+o(1)} \left(\left(\frac{H}{N^{\frac{3}{2}}}\right)^{\alpha} + N^{-\frac{\alpha}{k-1}}\right),
\]
as claimed.
\end{proof}

\begin{rem}\label{rem:alphaHalf}

It is worth highlighting the main novelty of our argument over the work in \cite[Lemma 8]{MV}, specifically the treatment of the expressions $G_{r,v}(\tau_1)$ in the notation of Proposition \ref{prop:GrvBd}. 

For simplicity, assume that $r = 1$ and, fixing $v$, write $A = 2^v$. In the context of \cite{MV} we may replace $\sgbsf$ with the set of divisors $t$ of the (squarefree) modulus $q$, in which case one may estimate the inner sum in $G_{1,v}(\tau_1)$ using the simple bound
\[
\sum_{\ss{a \sim A,\, a\mid q \\ (a,t) = 1 \\ a \equiv b \bmod{t}}} 1 \leq \sum_{\ss{a \sim A \\ a \equiv b \bmod{t}}} 1 \ll \frac{A}{t} + 1,
\]
\emph{uniformly} over both reduced residues $b \bmod{t}$, and over $t$. This follows from the equidistribution in residue classes of integers in an interval. While somewhat crude, this estimate suffices to produce the required power savings in $H$ in Montgomery and Vaughan's analogue of $\mc{E}$, as found in Proposition \ref{prop:e2 e3}.

In contrast, when $B$ is a sparse set of index $\alpha < 1$ we cannot reasonably hope for such equidistribution in residue classes in general. Even if, say, the optimistic bound
\[
\sum_{\ss{a \sim A, a \in \sgbsf \\ (a,t) = 1 \\ a \equiv b \bmod{t}}} 1 \ll (A t)^{o(1)} \left(\frac{A^{\alpha}}{t} + 1\right)
\]
held uniformly in $b \in (\mb{Z}/t\mb{Z})^{\times}$, tracing through the remainder of the proof of \cite[Lemma 8]{MV} we would find that the corresponding savings obtained is of the shape $H^{1/2 - \alpha +o(1)}$, and therefore only provides power savings for $\alpha > 1/2$. To deal with the most general situation (i.e., potentially with $\alpha \leq 1/2$ and no guarantee of equidistribution) we cannot simply rely on pointwise counts for elements of $\sgbsf$ in residue classes. The proof of Proposition \ref{prop:t4} demonstrates that power savings in $H$ may be obtained upon \emph{averaging} in both the residue class $b \bmod{t}$ and the modulus $t$.
\end{rem}

\section{The central limit theorem for general weights $\vp$}

\subsection{A proof of the central limit theorem}
We assemble the estimates of the previous sections to show that $C_k(H;\vp)$ and therefore $M_k(X,H;\vp)$ exhibit Gaussian behavior.

\begin{thm}\label{thm:C_general_gaussian}
Suppose $\vp$ is of bounded variation and supported in a compact subset of $[0,\infty)$. Assume moreover that $\vp$ is non-vanishing on some open interval. If $\sgb$ has index $\alpha \in (0,1)$ then
\[
C_k(H;\vp) =  C_2(H;\vp)^{\frac{k}{2}} \left(\mu_k + O(H^{-\frac{c}{k}})\right)
\]
for every positive integer $k$. Here $c$ is an absolute constant depending only on $\alpha$.
\end{thm}

\begin{proof}
By Lemma~\ref{lem:diag} and Proposition~\ref{prop:diag} we have
\begin{equation}
 C_k(H;\vp) = \mu_k C_2(H;\vp)^{\frac{k}{2}} \left( 1 + O\left(H^{-\alpha+o(1)}\right)\right) + O \bigg( \sum_{\substack{\mathbf{r} \text{ is repeated}\\\text{or non-paired}\\r_1,\ldots, r_k>1}}
 \prod_{i=1}^{k} \frac{\mobb^2(r_i)}{r_i}S(\mathbf{r}) \bigg).
 \end{equation}
Let $H^{8/9} \le N \le H$ and $M>1$. By Lemmas~\ref{lem:hn1}--\ref{lem:hn2} and Proposition~\ref{prop:e2 e3} we have
\begin{equation}
    \sum_{\substack{\mathbf{r} \text{ is repeated}\\\text{or non-paired}\\r_1,\ldots, r_k>1}}
 \prod_{i=1}^{k} \frac{\mobb^2(r_i)}{r_i}S(\mathbf{r}) \le M^{\alpha+o(1)}  +  H^{\frac{k-1}{2}+o(1)} \left(  N^{\frac{1}{2}} + H^{\frac{1}{2}}  N^{-\frac{1}{4(k-1)}} + \mc{E}^{\frac{1}{2}} H^{-\frac{1}{2}} \right) M^{\alpha-1+o(1)} .
\end{equation}
We apply Proposition~\ref{prop:t4} to bound $\mc{E}$, so that the right-hand side is
\[
\ll M^{\alpha + o(1)} \left( 1 + \frac{H^{\frac{k}{2}+o(1)}}{M}\left(\left(\frac{N}{H}\right)^{\frac{1}{2}} + \left(\frac{H}{N^{\frac{3}{2}}}\right)^{\frac{\alpha}{2}} + N^{-\frac{\min\{\alpha/2,1/4\}}{k-1}}\right)\right).
\]
We now choose $N=H^{1-c_1}$ and $M=H^{k/2-c_2/k}$ where $c_i$ are sufficiently small with respect to $\alpha$, and recall $C_2(H;\vp)=H^{\alpha+o(1)}$ by Lemma~\ref{lem:estimate c2rho}.
\end{proof}

Obviously this implies Theorem \ref{thm:main b} and thus the central limit theorem, Theorem \ref{thm:bdistGauss} for flat counts in short intervals. In fact more generally, combining Theorem \ref{thm:C_general_gaussian} with Proposition \ref{prop:general_limiting_moments} we see that as long as $H \leq X^{c_\alpha/k-\e}$, we have
\begin{equation}\label{eq:momExpB}
M_k(X,H;\vp) = C_2(H;\vp)^{\frac{k}{2}}\left( \mu_k + O(H^{-\frac{c}{k}})\right).
\end{equation}
By the moment method this implies that weighted counts also satisfy a central limit theorem.
\begin{thm}\label{thm:CLT_general_weights}
Let $H = H(X)$ satisfy \[H\ra \infty, \text{ yet }\,\,\frac{\log H}{\log X} \ra 0\text{ as }X \ra \infty,\] and choose a random integer $n \in [1,X]$ at uniform. Suppose that $\sgb$ is a regularly varying sequence of index $\alpha \in (0,1)$ and $\vp$ is a real-valued function of bounded variation and supported in a compact subset of $[0,\infty)$ and non-vanishing on some open interval. Then the random variable
\[
\frac{1}{\sqrt{C_2(H;\vp)}}\left( \sum_{u\in \Z} \vp\Big(\frac{u-n}{H}\Big)\bfree(u) - \meanb \sum_{h \in \Z}\vp\Big(\frac{h}{H}\Big) \right)
\]
tends to the standard normal distribution $N_\mb{R}(0,1)$ as $X\ra \infty$.
\end{thm}

\subsection{An application to long gaps} \label{subsec:longGaps}

Estimate \eqref{eq:momExpB} allows us to obtain strong information about the frequency of long gaps between consecutive $B$-free integers. Given $1 \leq H \leq X$, let
\[
\mc{G}(X,H) := \{ n \leq X : \Nbfree(n,H) = 0\}.
\]
Thus, $|\mc{G}(X,H)|$ records the number of length $H$ intervals with an endpoint $n \in [1,X]$ that contains no $B$-free numbers. 

Improving on work of Plaksin \cite{Plaksin1990}, Matom\"{a}ki \cite{Matomaki2012} used a sieve-theoretic method to show that for any $\varepsilon > 0$,
\[
|\mc{G}(X,H)| \ll XH^{-1+\varepsilon} \text{ for $1 \leq H \leq X^{\frac{1}{6}-\varepsilon}$}
\]
(with no upper bound constraint on the range of $H$ if $B$ consists only of primes). As a consequence of our $k$th-moment bounds, we can prove the following.
\begin{cor}\label{cor:EXhBds}
If $\sgb$ is of index $\alpha \in (0,1)$ then for any $k \geq 1$ and $1\leq H \leq X^{c_{\alpha}/k - \e}$ we have $|\mc{G}(X,H)|\ll_k XH^{-(2-\alpha) k}$. 
\end{cor}
Since we have $(2-\alpha)k \geq 1$ whenever $k \geq 1$ and $\alpha \in (0,1)$, our result improves on that of Matom\"{a}ki in some range of $H$. Note that, for instance, if $k = 2$ then $c_{\alpha}/k \geq 1/6$ whenever $0 < \alpha < (7-\sqrt{33})/2 = 0.6277...$, and our range contains hers, at least if $B$ does not consist only of primes.
\begin{proof}
If $n \in \mc{G}(X,H)$ then we of course have
\[
\left(\Nbfree(n,H) - \mc{M}_B H\right)^{2k} = (\meanb H)^{2k}.
\]
Combined with Proposition \ref{prop:index_variance} and Theorem \ref{thm:main b}, this implies that if $H \leq X^{c_{\alpha}/k-\e}$ then
\begin{align*}
\frac{|\mc{G}(X,H)|}{X} (\mc{M}_B H)^{2k} &\leq \frac{1}{X}\sum_{ n \leq X} \left(\Nbfree(n,H) - \meanb H\right)^{2k} = C_{2k,B}(H) + o(H^{k\alpha}) \\
&= \mu_{2k} C_2(H)^k + o(H^{k\alpha}) \ll_k H^{k\alpha}.
\end{align*}
We deduce immediately that $|\mc{G}(X,H)| \ll_k X H^{k\alpha- 2k} = X H^{-k(2-\alpha)}$, as claimed.
\end{proof}

\section{Fractional Brownian motion}\label{sec:fBm_proof}

\subsection{Convergence in $C[0,1]$}

We now prove Theorem \ref{thm:fBm_main}, showing that a random walk on the $B$-frees tends to a fractional Brownian motion. There is ready-made machinery to demonstrate that a sequence of random elements of $C[0,1]$, like $W_X(t)$ in \eqref{eq:W_def}, tend in distribution to a limiting element and we cite the relevant results here. (For a motivated exposition on convergence of random functions in $C[0,1]$, see for instance \cite{Billingsley99}.)

\begin{thm}\label{thm:finite_dim_and_tightness}
If $Y, Y_1, Y_2,\ldots$ are random elements of $C[0,1]$, then $Y_n$ converges in distribution to $Y$ as $n\rightarrow\infty$ as long as both of the following conditions are met:
\begin{enumerate}[label=(\roman*)]
\item (convergence of finite-dimensional distributions) for any $k\geq 1$ and $(t_1,\ldots,t_k) \in [0,1]^k$, we have convergence in distribution of the random vector
\[
(Y_n(t_1),\ldots,Y_n(t_k)) \rightarrow (Y(t_1),\ldots,Y(t_k)),
\]
\item (tightness) the sequence of random elements $Y_n$ of $C[0,1]$ is tight.\footnote{ That is for all $\e > 0$, there is a compact subset $K$ of $C[0,1]$ such that $\mb{P}(Y_n \notin K) < \e$ for all sufficiently large $n$. For $C[0,1]$ this is equivalent to the condition that $Y_n(0)$ is a tight family of real-valued random variables and $\lim_{\delta\rightarrow 0} \limsup_{n\rightarrow\infty}\mb{P}(\omega(Y_n, \delta) \geq \e) = 0$, where the modulus of continuity of a function $f \in C[0,1]$ is given by $\omega(f,\delta) = \sup_{|s-t|\leq \delta} |f(s)-f(t)|;$ see \cite[Theorem 7.3]{Billingsley99}.}
\end{enumerate}
\end{thm}

\begin{proof}
This is a direct consequence of \cite[Lemma 16.2 and Theorem 16.3]{Kallenberg}.
\end{proof}

We also have the following device for proving tightness.

\begin{thm}[Kolmogorov--Chentsov]\label{thm:Kolmogorov_Chentsov}
Using the notation of the previous theorem, if $Y_n(0) = 0$ for all $n$ and if there are an absolute constant $C$ and constants $a,b > 0$ such that
\[
\sup_n\, \E\, |Y_n(s) - Y_n(t)|^a \leq C|s-t|^{1+b},
\]
for all $s,t \in [0,1]$, then the sequence of random elements $Y_n$ of $C[0,1]$ is tight.
\end{thm}

\begin{proof}
This is a special case of \cite[Corollary 16.9]{Kallenberg}.
\end{proof}

\subsection{The $B$-free random walk}

We apply these results to the random functions $W_X(t)$ with $X\ra\infty$. We need one last lemma regarding regularly varying sequences.

\begin{lem}\label{lem:upperbound_regularvar}
If we have a sequence of natural numbers $J$ is regularly varying with index $\alpha \in (0,1)$ then for any $\e > 0$,
\[
\frac{N_J(tH)}{N_J(H)} \ll_\e t^{\alpha-\e},
\]
for all $t \leq 1$ and $H$ larger than the first element of $J$.
\end{lem}

\begin{proof}
Obviously the result is true if $tH < 1$, so suppose $tH \geq 1$.

If $J \subset \N$ is regularly varying with index $\alpha \in (0,1)$, then there is some slowly varying $L$ such that $N_J(x) \ll x^\alpha L(x)$ for all $x$ and $N_J(x) \gg x^\alpha L(x)$ for all $x$ larger than the first element of $J$. For convenience we will take $L(x)$ to be defined for all $x \geq 1$ with \[\inf_{x \in [1,K]} L(x) > 0\text{ for any }K > 0.\]
(The reader should check that this may be done.)

It then follows from Karamata's representation of slowly varying functions (see \cite[consequence (2.5) of Theorem 2.2 in p. 180]{Korevaar}) that if $t\leq 1$, 
\[
\frac{L(tH)}{L(H)} \leq 2 t^{-\e},
\]
for $tH$ and $H$ sufficiently large depending on $\e$. By compactness then
\[
\frac{L(tH)}{L(H)} \ll_\e t^{-\e}
\]
for $tH, H \geq 1$, which implies the claim.
\end{proof}

\begin{proof}[Proof of Theorem \ref{thm:fBm_main}]
Let $\{Z(t):\, t \in [0,1]\}$ be a fractional Brownian motion with Hurst parameter $\alpha/2$. By Theorem \ref{thm:finite_dim_and_tightness} we need only demonstrate (i) the finite dimensional distributions of $W_X(t)$ tend to those of $Z(t)$ and (ii) tightness for the family $W_X$. 

Let us treat (i) first. Note for each $t \in [0,1]$ we have $\mb{E}\, W_X(t) \ra 0$ as $X \ra \infty$ (which agrees of course with $\mb{E}\, Z(t) = 0$). Moreover, for $s,t \in [0,1]$ with $s > t$,
\begin{align*}
\E\, |W_X(s) - W_X(t)|^2 &= \frac{1}{A_\alpha N_{\sgb}(H)} \frac{1}{X} \sum_{n = \lfloor tH \rfloor}^{X + \lfloor sH \rfloor} |\Nbfree(n,\lfloor (s-t)H \rfloor) - \mc{M}_B \lfloor (s-t)H \rfloor + O(1) |^2 \\
&\sim \frac{N_{\sgb}((s-t)H)}{N_{\sgb}(H)} \sim (s-t)^\alpha,
\end{align*}
by Proposition \ref{prop:regularlyvarying_variance}. As $W_X(0) = 0$, this allows one to deduce that $W_X(t)$ has the same limiting covariance function as $Z(t)$.

Thus we will have the finite dimensional distributions of $W_X(t)$ tend to those of $Z(t)$ if we show for any $k\geq 1$ and any $t_1,\ldots,t_k$ that the vector $(W_X(t_1),\ldots,W_X(t_k))$ tends to a Gaussian vector. By the Cram\'er--Wold device \cite[Theorem 29.4]{Billingsley1995} this will be true if for any fixed real numbers $\theta_1,\ldots,\theta_k$ the random variable
\[
\sum_{j=1}^k \theta_j W_X(t_j) = \frac{1}{\sqrt{A_\alpha N_{\sgb}(H)}} \sum_{j=1}^k \theta_j Q(t_j\cdot H)
\]
tends to a real valued Gaussian distribution. But
\[
\sum_{j=1}^k \theta_j Q(t_j\cdot H) = \sum_{u \in \Z} \vp\Big(\frac{u-n}{H}\Big) \bfree(u) - \meanb \sum_{h \in \Z} \vp\Big(\frac{h}{H}\Big) + O(1),\quad \textrm{for}\quad \vp = \sum \theta_j \mathbf{1}_{(0,t_j]},
\]
and so the Gaussian behavior follows from Theorem \ref{thm:CLT_general_weights}.

We now demonstrate (ii). Note that for any positive integer $\nu$ and $\e > 0$, we have
\begin{align*}
\E\, |W_X(s) - W_X(t)|^{2\nu} &= \frac{1}{(A_\alpha N_{\sgb}(H))^{\nu}} \frac{1}{X} \sum_{n = \lfloor tH \rfloor}^{X + \lfloor sH \rfloor} |\Nbfree(n,\lfloor (s-t)H \rfloor) - \mc{M}_B \lfloor (s-t)H \rfloor + O(1) |^{2\nu} \\
&\ll_\nu \frac{(N_{\sgb}((s-t)H)+1)^\nu}{N_{\sgb}(H)^\nu}  \ll_{\nu, \e} |s-t|^{(\alpha-\e) \nu},
\end{align*}
as long as $X$ (and thus $H$) is sufficiently large so that $N_{\sgb}(H)$ is non-zero, using Lemma \ref{lem:upperbound_regularvar} in the last step.
Hence choosing $\e$ smaller than $\alpha$ and $\nu$ large enough that $(\alpha -\e) \nu > 1$, condition (ii) follows from the Kolmogorov--Chentsov Theorem. This completes the proof.
\end{proof}

\bibliography{Sfree}
\bibliographystyle{plain}

\Addresses

\end{document}